\documentclass[11pt]{amsart}

\newcommand{\editdate}{August 1, 2018}

\newcommand{\Lra}{\Longrightarrow}

\usepackage{amsmath,amssymb,amscd,amsfonts,verbatim}
\usepackage[margin=1in]{geometry}
\usepackage{rotating}
\usepackage{color}
\usepackage[all,cmtip]{xy}
\makeatother\newcommand{\lesarrow}[1][{}]{\ar@`{[r],[d],[lll],[dlll]}[dll]|{#1}}\makeatletter
\usepackage{pifont}
\usepackage{mathrsfs}  
\usepackage{flafter}

\usepackage{tikz}
\usetikzlibrary{matrix,arrows}

\newtheorem{proposition}[equation]{Proposition}
\newtheorem{theorem}[equation]{Theorem}
\newtheorem{corollary}[equation]{Corollary}
\newtheorem{lemma}[equation]{Lemma}

\theoremstyle{definition}

\newtheorem{definitions}[equation]{Definitions}
\newtheorem{remark}[equation]{Remark}
\newtheorem{remarks}[equation]{Remarks}
\newtheorem{example}[equation]{Example}

\newtheorem{question}[equation]{Question}
\newtheorem{inductive-hypotheses}[equation]{Inductive Hypotheses}

\newcommand{\C}{\mathbb{C}} 
\newcommand{\N}{\mathbb{N}} 
 
\newcommand{\R}{\mathbb{R}}
\newcommand{\Z}{\mathbb{Z}} 
\newcommand{\acts}{\mathbin{\raisebox{-.5pt}{\reflectbox{\begin{sideways}
$\circlearrowleft$\end{sideways}}}}}

\DeclareMathOperator{\Fred}{Fred}

\numberwithin{figure}{section}
\numberwithin{table}{section}
\numberwithin{equation}{section}

\makeatletter
\let\c@equation\c@figure
\makeatother

\makeatletter
\let\c@table\c@figure
\makeatother

\makeatletter
\let\c@algorithm\c@figure
\makeatother

\newcommand{\ouremph}[1]{{\bf #1}}

\begin{document}

\title[Mayer-Vietoris sequences and equivariant $K$-theory rings of toric varieties]{Mayer-Vietoris sequences and equivariant\\ $K$-theory rings of toric varieties\\ \editdate}

\author{Tara S. Holm} 
\address{Department of Mathematics, Cornell University, Ithaca, New York 14853-4201, USA}
\email{tsh@math.cornell.edu}
\thanks{TH was partially supported by the Simons Foundation through Grants \#266377 and \#79064 and by the National Science Foundation through Grant \#DMS--1711317.}

\author{Gareth Williams} 
\address{School of Mathematics and Statistics, The Open University, Walton Hall, Milton Keynes MK7~6AA, UK} 
\email{g.r.williams@open.ac.uk}
\thanks{GW was partially supported by a Research in Pairs grant from the London Mathematical Society.}

\keywords{Toric variety, fan, equivariant $K$-theory, piecewise Laurent polynomial}
\subjclass[2010]{Primary: 19L47; Secondary: 55N15, 55N91, 14M25, 57R18}

\begin{abstract}
We apply a Mayer-Vietoris sequence argument to identify the Atiyah-Segal 
equivariant complex $K$-theory rings of certain toric varieties with rings of 
integral piecewise Laurent polynomials on the associated fans. We provide 
necessary and sufficient conditions for this identification to hold for toric 
varieties of complex dimension $2$, including smooth and singular cases. 
We prove that it always holds for smooth toric varieties, regardless of whether 
or not the fan is polytopal or complete. Finally, we introduce the notion of fans 
with ``distant singular cones," and prove that the identification holds for them. 
The identification has already been made by Hararda, Holm, Ray and Williams 
in the case of divisive weighted projective spaces; in addition to enlarging the 
class of toric varieties for which the identification holds, this work provides an 
example in which the identification fails. We make every effort to ensure that 
our work is rich in examples.
\end{abstract}

\maketitle

\section{Background and notation}\label{section:background}

Toric varieties are an important class of examples in symplectic and algebraic geometry.  Their explicit definition and combinatorial properties mean that their invariants are amenable to direct calculation. They are an important testing ground for conjectures and theories. In this paper, we use elementary tools to explore the topological equivariant $K$-theory rings of toric varieties. The goal is to find a large class of toric varieties for which we may identify this $K$-theory with rings of piecewise Laurent polynomials. We begin with a quick overview of where our work fits in the current literature.

Let $G$ be a compact Lie group and $G\acts Y$ a $G$-space, which we commonly abbreviate to $Y$. Our aim is to consider the $G$-equivariant complex $K$-theory rings $K^*_G(Y)$ for certain $Y$ in the case when $G$ is a torus. This work may be considered as a broadening of the results in \cite{HHRW}, though it is not a direct extension. In part, this is because there are $Y$ for which results of \cite{HHRW} but not the present paper apply, and yet other $Y$ for which the results of the present paper but not \cite{HHRW} apply (though there is a large family of $Y$, namely smooth, polytopal toric varieties, for which the results of both papers apply). The main tool of the current paper, namely the Mayer-Vietoris sequence, is fundamentally different from, and simpler than, the techniques developed in \cite{HHRW}. The other notable difference is that the present paper is concerned solely with equivariant $K$-theory, whereas \cite{HHRW} is one of a number of papers \cite{BFR,BNSS,BSS,MF} to consider other complex oriented equivariant cohomology theories.

Given the plurality of $K$-theory functors and results for algebraic vector bundles and algebraic $K$-theory, it is important to keep in mind precisely which $K$-theory rings we consider. We are concerned with the unreduced Atiyah-Segal $G$-equivariant ring $K^*_G(Y)$ \cite{Seg68}, graded over the integers. For compact $Y$, $K_G^0(Y)$ is constructed from equivalence classes of $G$-equivariant complex vector bundles; otherwise, it is given by equivariant homotopy classes $[Y,\Fred(\mathcal{H}_G)]_G$, where $\mathcal{H}_G$ is a Hilbert space containing infinitely many copies of each irreducible representation of $G$ \cite{AtiSeg04}. For the $1$-point space $*$ with trivial $G$-action, we write the \ouremph{coefficient ring} $K_G^*(*)$ as $K_G^*$. It is isomorphic to $R(G)[z,z^{-1}]$, where $R(G)$ denotes the complex representation ring of $G$, and realises $K_G^0$; the \ouremph{Bott periodicity element} $z$ has cohomological dimension $-2$. The equivariant projection $Y\to *$ induces the structure of a graded $K_G^*$-algebra on $K^*_G(Y)$, for any $G\acts Y$.

We consider $G\acts Y$ in the case that $Y$ is a toric variety and $G$ is a suitable torus. Specifically, we consider the $2n$-dimensional toric variety  $X_\Sigma$ where $\Sigma$ is a \ouremph{fan} in $N_{\R} =N\otimes_{\Z}\R \cong \R^n$, with respect to the lattice $N$. The compact $n$-torus $T=S^1\times\cdots\times S^1 = (N\otimes_{\Z}\R)/N$ acts on $X_\Sigma$. In the context of this $G\acts Y$, there is much \cite{anpa:okt, bri:spa,dup:fpv,kan:ctp,kly:ebt,liya:ste,mor:ktt,pay:tvb,vevi:hak} in the literature regarding algebraic bundles, results in algebraic and operational $K$-theory, and the relationships between these results. For example Vezzosi and Vistoli \cite{vevi:hak} computed equivariant algebraic $K$-theory for smooth toric varieties, and the comparison theorem of Thomason \cite{Thomason} provides a link between their answer and the topological {\bf Borel} equivariant $K$-theory. The latter is the completion of the topological Atiyah-Segal equivariant $K$-theory; however as there is no known method to reverse the process of completion, knowledge of the topological Borel equivariant $K$-theory does not guarantee results about topological Atiyah-Segal equivariant $K$-theory. Thus it is important to keep in mind that the present work focuses on {\bf topological, Atiyah-Segal} equivariant $K$-theory, and that we consider a {\bf topological} invariant of varieties arising in algebraic {geometry} (endowed with the classical topology). Further details regarding the relationship between topological Borel equivariant $K$-theory and topological Atiyah-Segal equivariant $K$-theory, in the context of toric varieties, may be found in \cite[\S6]{HHRW}.

The correspondence between fan and toric variety is an elegant interplay which is crucial to our work. Fans are constructed from \ouremph{cones} in a highly controlled manner. We assume that $\Sigma$ has finitely many cones, each of which is strongly convex and rational. Then we have affine pieces $U_{\sigma}$ for each cone $\sigma\in\Sigma$, and
\begin{equation}\label{equation:full-cover}
X_\Sigma=\bigcup\limits_{\sigma\in\Sigma} U_\sigma.
\end{equation}
We note that a single cone $\sigma$ may be viewed as a fan itself. As a fan, its cones are $\sigma$ and all the faces of $\sigma$. We will henceforth abuse notation and write $\sigma$ for both the cone and the fan. Occasionaly it will be convenient to consider a fan as a subspace of Euclidean space rather than as a collection of cones. For definitions and details relating to cones, fans and toric varieties, the reader is recommended to consult one or more of \cite{ful:itv}, \cite{colisc:tv} and \cite{oda:cba}.

Before proceeding further, it is useful to fix notation and conventions; our aim is to be as consistent as possible with \cite{HHRW}. In the case of a single (possibly singular) cone $\sigma$, we have an equivariant homotopy equivalence
\[
U_\sigma\simeq_{T^n}T^n/T_\sigma
\]
by \cite[Proposition 12.1.9, Lemma 3.2.5]{colisc:tv}, where $T_\sigma$ is the \ouremph{isotropy torus} of $U_\sigma$.  (Note that in \cite{colisc:tv}, the authors
work with algebraic tori $(\C^*)^n$; as we are only concerned with homotopy equivalence, we have
abusively used the same notation for the compact tori.) 
As is explained further in \cite[\S4]{HHRW}, we may write 
\begin{equation}\label{equation:one-cone}
K^*_{T^n}(U_\sigma)\cong K^*_{T^n}(T^n/T_\sigma)\cong P_K(\sigma),
\end{equation}
where (as discussed in \cite[Example~4.12]{HHRW}), $P_K(\sigma)$ is a graded ring which additively includes a copy of $\Z[{\alpha}_1^{\pm1},\dots,{\alpha}_n^{\pm1}]/J_\sigma$ in each even degree, and  zero in each odd degree. Here the classes $\alpha_i^{\pm1}$ are in cohomological degree zero and $J_\sigma$  denotes the ideal generated by certain Euler classes, which we describe below.

For each cone $\sigma$, we may define a subspace $\sigma^\perp$
of the dual space $M_{\R} = M\otimes_{\Z}\R$ by
\[
\sigma^\perp=\left\{ m\in M_{\R} \ \Big|\  \langle m,u\rangle= 0\ \forall u\in\sigma\right\}.
\]
When $\sigma$ is $d$-dimensional, $\sigma^\perp$ is $(n-d)$-dimensional. Because $\sigma$ is rational, $\sigma^\perp\cap M$ is a rank $(n-d)$ sublattice in the dual lattice $M$.  We may choose a $\Z$-basis of this sublattice, $\nu_1,\dots,\nu_{n-d}$, and for each $\nu_j$, there is a $K$-theoretic equivariant Euler class  $e(\nu_j) = (1-{\boldsymbol\alpha}^{\nu_j})$. We 
define 
\[
J_\sigma := \bigg\langle \left(1-{\boldsymbol\alpha}^{\nu_1}\right)\ ,\ \dots\ ,\ 
\left(1-{\boldsymbol\alpha}^{\nu_{n-d}}\right)\bigg\rangle 
< \Z[{\alpha}_1^{\pm 1},\dots,{\alpha}_n^{\pm 1}].
\]

\begin{remark}\label{remark:equivalence reps}
For an equivalence class $[f_\sigma]\in\Z[{\alpha}_1^{\pm 1},\dots,{\alpha}_n^{\pm 1}]/J_\sigma$, we shall usually work with a choice of representative $f_\sigma\in \Z[{\alpha}_1^{\pm 1},\dots,{\alpha}_n^{\pm 1}]$. One consequence of this is that when $\sigma$ has dimension $d$ and is considered in an ambient space of dimension $n>d$, we may make a choice of representative $f_\sigma$ for any $[f_\sigma]\in\Z[{\alpha}_1^{\pm 1},\dots,{\alpha}_n^{\pm 1}]/J_\sigma$ involving only those $\alpha_i^{\pm1}$ which do not arise in the definition of $J_\sigma$. In other words, considering $\sigma$ in a larger ambient space has the effect of introducing more variables in $\Z[\alpha_1^{\pm1},\ldots,\alpha_n^{\pm1}]$ but also more relations in $J_\sigma$, and hence no practical effect overall.
\end{remark}

More generally, a definition of $P_E(\Sigma)$ for any complex-oriented equivariant cohomology theory $E$ is given in \cite[Definition~4.6]{HHRW} as the limit of a diagram in an appropriate category. In this paper, we work solely in the case when $E$ is equivariant $K$-theory, and we interpret $P_K(\Sigma)$ in the same way as in \cite[Example~4.12]{HHRW}: each element of $P_K(\Sigma)$ may be interpreted as the equivalence class of an integral \ouremph{piecewise Laurent polynomial} on the fan $\Sigma$. A piecewise Laurent polynomial on $\Sigma$ is determined by its values on the maximal cones, and the ring of piecewise Laurent polynomials on $\Sigma$ is denoted $PLP(\Sigma)$. In this interpretation, as a graded ring $P_K(\Sigma)$ is zero in odd degrees, $PLP(\Sigma)$ in even degrees, and addition/multiplication of classes corresponds to cone-wise addition/multiplication of Laurent polynomials.

Recall from \eqref{equation:one-cone} that, in the case of a single cone $\sigma$, $K^*_T(X_\sigma)\cong P_K(\sigma)$. A natural question to ask is:
\begin{question}\label{question}
For which fans $\Sigma$ is $K_{T^n}^*(X_\Sigma)\cong P_K(\Sigma)$?
\end{question}
\noindent In Section \ref{section:mv}, we set up our main tool to analyse Question \ref{question}. In Sections \ref{section:R^2} and \ref{section:R^2-complete} we give a satisfying answer (Theorem \ref{theorem:R^2}) to Question \ref{question} for fans in $\R^2$. In particular, we show that fans $\Sigma$ in $\R^2$ corresponding to Hirzebruch surfaces or weighted projective spaces do satisfy $K_{T^n}^*(X_\Sigma)\cong P_K(\Sigma)$, but fans in $\R^2$ corresponding to those fake weighted projective spaces which are not weighted projective spaces, do not (Examples \ref{example:fwps} and \ref{example:Hirzebruch}). In Sections \ref{section:single-cone} and \ref{section:smooth} we develop theory for smooth fans. Smooth {\bf polytopal} fans were considered in \cite{HHRW} but the present treatment does not require polytopal, and culminates in the expected answer (Theorem \ref{theorem:smooth fans}) to Question \ref{question} for smooth fans. Finally, in Section \ref{section:distant} we introduce the notion of fans with \ouremph{distant singular cones} and analyse Question \ref{question} in this context, providing several examples in $\R^3$. We emphasize that our computations involve solely elementary techniques and do not rely upon any of the sophisticated machinery which is prominent in much of the literature.\\

\noindent {\bf Acknowledgements.} We are especially grateful to Nige Ray for many hours of discussion on the topology of toric varieties and the combinatorics of fans; and to Mike Stillman and Farbod Shokrieh for kindly discussing and pointing out references for the 
algebraic results described in Appendix~\ref{section:lil}. We are also grateful to the referee for constructive comments which have improved our work.

\section{Mayer-Vietoris}\label{section:mv}

We aim to use a Mayer-Vietoris argument to compute the Atiyah-Segal \cite{Seg68} equivariant $K$-theory ring $K^*_T(X_\Sigma)$ of a toric variety $X_\Sigma$. We set up the Mayer-Vietoris sequence as follows. Suppose $\Delta'$ and $\Delta''$ are both sub-fans of $\Sigma$ such that
$\Sigma=\Delta'\cup\Delta''$. We call such a union a \ouremph{splitting of $\Sigma$}.
For every splitting, there is a Mayer-Vietoris long exact sequence of $K^*_T$-algebras:

\begin{equation}\label{equation:mv}
\begin{array}{c}
\begin{tikzpicture}[descr/.style={fill=white,inner sep=1.5pt}]
        \matrix (m) [
            matrix of math nodes,
            row sep=1em,
            column sep=2.5em,
            text height=1.5ex, text depth=0.25ex
        ]
        { \phantom{\cdots} & & \cdots & K^{2i-1}_T(X_{\Delta'\cap\Delta''}) \\
            & K^{2i}_T(X_\Sigma)& K^{2i}_T(X_{\Delta'})\oplus K^{2i}_T(X_{\Delta''}) & K^{2i}_T(X_{\Delta'\cap\Delta''})  \\
            &  K^{2i+1}_T(X_\Sigma) & \cdots \ .& \\
        };
        \path[overlay,->,>=to,anchor=south,font=\scriptsize]
       (m-1-3) edge (m-1-4)
        (m-1-4) edge[out=355,in=175,black] (m-2-2) 
        (m-2-2) edge (m-2-3)
        (m-2-3) edge node {\#} (m-2-4)
        (m-2-4) edge[out=355,in=175,black] (m-3-2)
        (m-3-2) edge (m-3-3);
\end{tikzpicture}
\end{array}
\end{equation}
\noindent We say a splitting $\Sigma=\Delta'\cup\Delta''$ is \ouremph{proper}  if 
$\Sigma \neq \Delta'$ and $\Sigma\neq\Delta''$. The only fans which do not admit proper splittings are those fans which consist of a single cone (and all its faces). Since every fan contains the zero cone, $\Delta'\cap\Delta''$ is never empty. Whilst the union or intersection of two fans in general need not be a fan, $\Delta'\cup\Delta''$ and $\Delta'\cap\Delta''$ {\em are} fans here because $\Delta'$ and $\Delta''$ are both sub-fans of the {\em same} fan.

Typically we will work with splittings $\Sigma=\Delta'\cup\Delta''$
where we have control over $\Delta'$, $\Delta''$ and $\Delta'\cap\Delta''$.  In particular, 
when $K^{2i+1}_T(X_\Gamma)=0$ for $\Gamma=\Delta',\Delta'',\Delta'\cap\Delta''$ (all $i\in\Z$), the long exact sequence 
\eqref{equation:mv} becomes the $4$-term exact sequence in the top row of the diagram below.

\begin{xy}
\xymatrix{
0\ar[r]& K^{2i}_T(X_{\Sigma})\ar[r]&K^{2i}_T(X_{\Delta'})\oplus K^{2i}_T(X_{\Delta''})\ar[r]^-{\#}\ar[d]^\cong& K^{2i}_T(X_{\Delta'\cap\Delta''})\ar[r]\ar[d]^\cong& K^{2i+1}_T(X_{\Sigma})\ar[r]& 0\\
&&PLP(\Delta')\oplus PLP(\Delta'')\ar[r]^-{\#}&PLP(\Delta'\cap\Delta'')
}
\end{xy}

\noindent We shall work in situations where the vertical maps are isomorphisms of $K^*_T$-algebras and thus we treat the four terms exact sequence as

\begin{equation}\label{equation:4-term}
0\to K^{2i}_T(X_{\Sigma})\longrightarrow PLP(\Delta')\oplus PLP(\Delta'')\stackrel{\#}
{\longrightarrow} PLP(\Delta'\cap\Delta'')\longrightarrow K^{2i+1}_T(X_{\Sigma})\to 0,
\end{equation}

\noindent and deduce that we have $K^{2i}_T(X_\Sigma)\cong PLP(\Sigma)$ as $K^*_T$-algebras.

The exact sequence \eqref{equation:4-term} makes plain the central role played by the map 
\[
PLP(\Delta')\oplus PLP(\Delta'')\stackrel{\#}{\to}PLP(\Delta'\cap\Delta''),
\] 
a map which may be made very explicit. For $(F,G)\in PLP(\Delta')\oplus PLP(\Delta'')$,
\begin{equation}\label{equation:sharp}
\#\big( (F,G)\big) = F|_{\Delta'\cap \Delta''}-G|_{\Delta'\cap \Delta''}\in PLP(\Delta'\cap\Delta'').
\end{equation}
It is immediate from \eqref{equation:4-term} that we have $K^0_T(X_{\Sigma}) = \ker(\#) = PLP(\Sigma)$, and $K^1_T(X_{\Sigma}) = \mathrm{coker}(\#)$.

\section{General results for fans in $\R^2$}\label{section:R^2}

Fans in $\R^2$ are amenable to study because of the limited ways in which their cones may interact with each other. If $\Sigma$ is a fan in $\R^2$ with cones $\sigma$ and $\tau$, then $\sigma\cap\tau$ must be either $\{0\}$, a ray, or the union of two rays. Only the last case requires serious consideration, and by careful construction of the Mayer-Vietoris argument we may limit the frequency of this occurrence to just once when the fan is complete, and not at all when the fan is incomplete.

When the fan is a single cone $\sigma$, then \eqref{equation:one-cone} guarantees that $K^*_T(X_\sigma)\cong P_K(\sigma)$. We must now consider the possibility of more than one cone. We say that a non-trivial incomplete fan $\Sigma$ in $\R^2$ is a \ouremph{clump} if $\Sigma\setminus\{0\}$ is connected (as a subspace of $\R^2$). We begin our analysis by proving that clumps always satisfy $K^*_T(X_\Sigma)\cong P_K(\Sigma)$.

\begin{lemma}\label{lemma:consecutive-2-cones}
If $\Sigma$ is a clump, then $K^*_T(X_\Sigma)\cong P_K(\Sigma)$.
\end{lemma}

\begin{proof}If $\Sigma$ has no two-dimensional cones, then it must consist of a single zero- or one-dimensional cone and the lemma is immediate from \eqref{equation:one-cone}. We now consider the case that $\Sigma$ has $k>0$ two-dimensional cones, and proceed by induction on $k$. If $k=1$, the lemma is immediate from \eqref{equation:one-cone}, which concludes the base case. Now suppose that $\Sigma$ has two-dimensional cones $\sigma_1,\ldots,\sigma_k$ and rays $\rho_1,\ldots,\rho_{k+1}$, indexed so that $\sigma_i\cap\sigma_{i+1}=\rho_{i+1}$ for $1\leq i\leq k-1$. We assume inductively that $K^*_T(X_{\sigma_1\cup\cdots\cup\sigma_{k-1}})\cong P_K(\sigma_1\cup\cdots\cup\sigma_{k-1})$ as a $K^*_T$-algebra.

In the Mayer-Vietoris sequence (\ref{equation:mv}) take $\Delta'=\sigma_1\cup\cdots\cup\sigma_{k-1}$ and $\Delta''=\sigma_k$. Then $\Delta'\cap\Delta''=\rho_k$ and we have $K^*_T(X_{\Delta'\cap\Delta''})\cong P_K(\rho_k)$ as a $K^*_T$-algebra. Then the Mayer-Vietoris sequence splits
into a $4$-term sequence as in \eqref{equation:4-term},
\[
0\to K^0_T(X_{\Sigma})\longrightarrow PLP(\Delta')\oplus PLP(\Delta'')\stackrel{\#}
{\longrightarrow} PLP(\Delta'\cap\Delta'')\longrightarrow K^1_T(X_{\Sigma})\to 0.
\]
Noting that $\#$ is surjective from the second summand, we have $K^0_T(X_{\Sigma}) = \ker(\#) = PLP(\Sigma)$ and $K^1_T(X_{\Sigma}) = \mathrm{coker}(\#) = 0$. Since all the identifications made in the inductive step were as $K^*_T$-algebras, we may assemble the $4$-term sequences to achieve an algebra isomorphism $K^*_T(X_\Sigma) \cong P_K(\Sigma)$, as desired. This completes the inductive step, and the lemma follows.
\end{proof}

We are now able to analyze incomplete fans in $\R^2$.

\begin{lemma}\label{lemma:incomplete}
If $\Sigma$ is an incomplete fan in $\R^2$, then $K^*_T(X_\Sigma)\cong P_K(\Sigma)$.
\end{lemma}

\begin{proof}
Since every incomplete fan in $\R^2$ may be decomposed as a union of clumps with pairwise 
intersections equal to $\{0\}$, it suffices to prove the lemma for such unions of clumps. We proceed by induction on the number of clumps. 
Lemma~\ref{lemma:consecutive-2-cones} establishes the base case.

Suppose inductively that the lemma holds for such unions of fewer than $k$ clumps, and consider $\Sigma=\Sigma_1\cup\cdots\cup\Sigma_k$, a union of $k$ clumps satisfying $\Sigma_i\cap\Sigma_j = \{0\}$ for $i\neq j$. In the Mayer-Vietoris sequence (\ref{equation:mv}), take $\Delta'=\Sigma_1\cup\cdots\cup\Sigma_{k-1}$ and $\Delta''=\Sigma_k$. Since the pairwise intersections $\Sigma_i\cap\Sigma_j$ are $\{0\}$, we have $\Delta'\cap\Delta''=\{0\}$ and $K^*_T(X_{\Delta'\cap\Delta''})\cong P_K(\{0\})$ as a $K^*_T$-algebra. The Mayer-Vietoris sequence becomes a $4$-term sequence as in \eqref{equation:4-term},
\[
0\to K^0_T(X_{\Sigma})\longrightarrow PLP(\Delta')\oplus PLP(\Delta'')\stackrel{\#}
{\longrightarrow} PLP(\Delta'\cap\Delta'')\longrightarrow K^1_T(X_{\Sigma})\to 0.
\]
Noting that $\#$ is surjective from the second summand by Lemma~\ref{lemma:consecutive-2-cones}, we have
$K^0_T(X_{\Sigma}) = \ker(\#) = PLP(\Sigma)$ and 
$K^1_T(X_{\Sigma}) = \mathrm{coker}(\#) = 0$. 

Since all the identifications made in the inductive step were as $K^*_T$-algebras, we may assemble the $4$-term sequences to achieve an algebra isomorphism $K^*_T(X_\Sigma) \cong P_K(\Sigma)$, as desired. This completes the inductive step, and the lemma follows.
\end{proof}

\noindent Lemma~\ref{lemma:incomplete} gives a description of the equivariant $K$-theory ring of a toric
variety whose fan in $\R^2$ is incomplete. To continue our study, we analyze the map $\#$ in more detail in the context of a clump.

Let $\Sigma$ be a clump. Denote the maximal cones of $\Sigma$ by $\sigma_1,\ldots,\sigma_k$ and the rays by $\rho_1,\ldots,\rho_{k+1}$, numbered so that 
$\sigma_i\cap\sigma_{i+1}=\rho_{i+1}$ for $1\leq i\leq k-1$ (as in Figure \ref{figure:clump}). We consider surjectivity of the map
\[
\#\colon PLP(\Sigma)\to PLP(\rho_1\cup\rho_{k+1}),
\]
interpreting $\#$ as the restriction of a piecewise Laurent polynomial on $\Sigma$, to a piecewise Laurent polynomial on the subfan $\rho_1\cup\rho_{k+1}\subset \Sigma$. Our use of the symbol $\#$ here is a small abuse of notation, which we justify upon anticipation of the application!

\begin{figure}[h]
\includegraphics[height=2in]{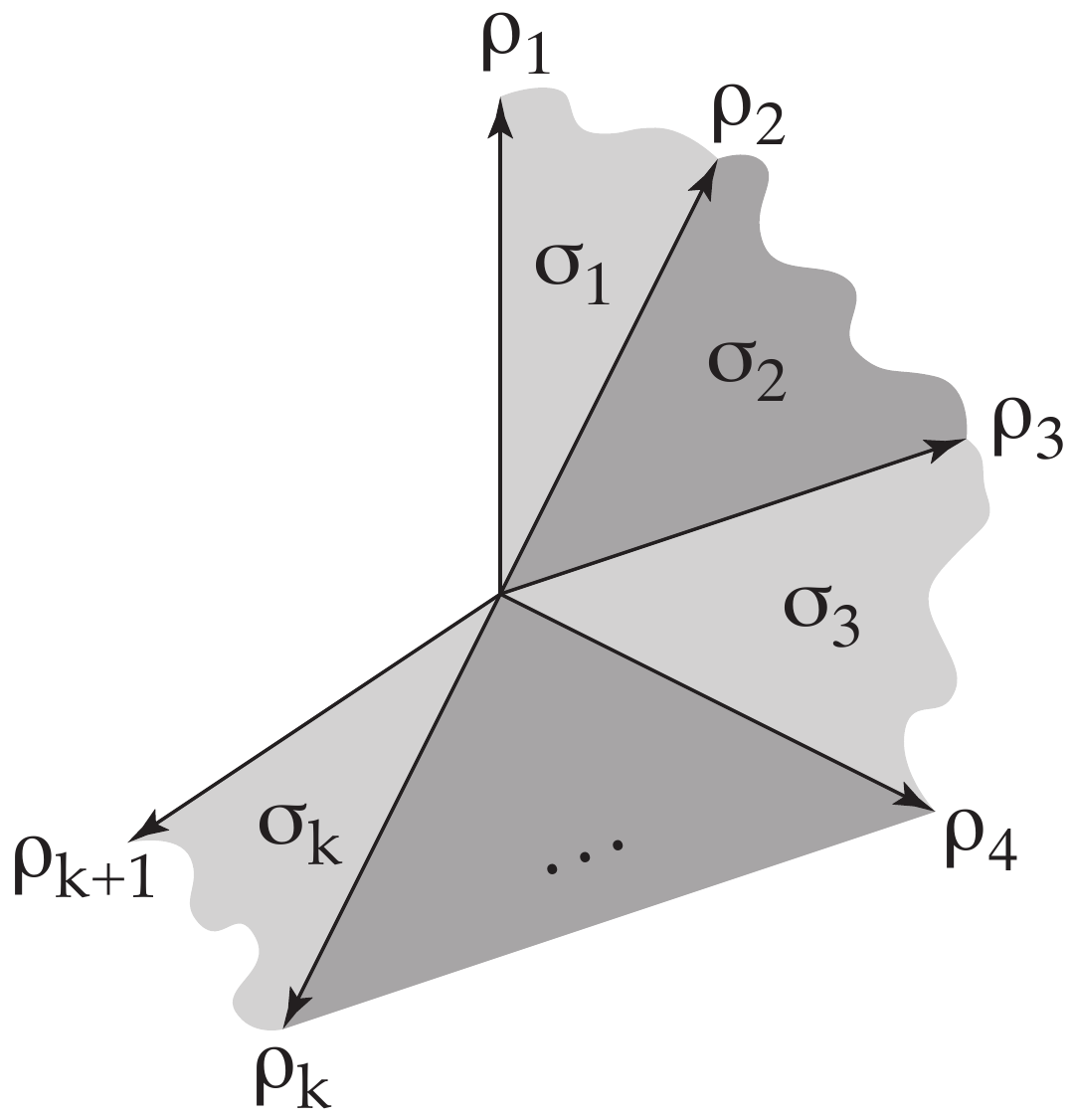}\hspace*{1em}\includegraphics[height=2in]{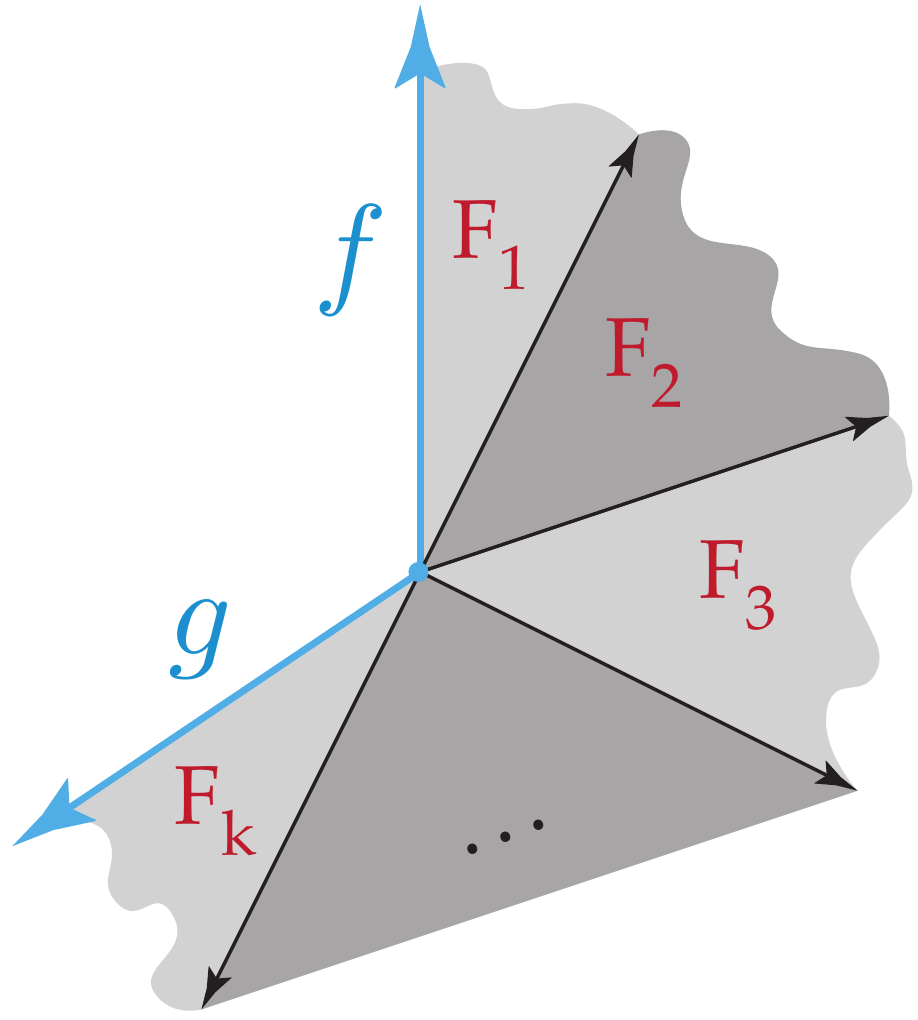}
\caption{An incomplete fan in $\R^2$ which is a clump (left) and related piecewise Laurent polynomials (right).\label{figure:clump}}
\end{figure}

The question of whether $\#\colon PLP(\Sigma)\to PLP(\rho_1\cup\rho_{k+1})$ is surjective becomes: given $(f,g)$ piecewise on $\rho_1\cup\rho_{k+1}$, does there exist $(F_1,\ldots,F_k)\in PLP(\Sigma)$ whose image under $\#$ is $(f,g)$? It is convenient to phrase the piecewise conditions in terms of ideal membership. Observe that $(f,g)$ being piecewise on $\rho_1\cup\rho_{k+1}$ is equivalent to $f-g\in J_{\{0\}}$. The tuple $(F_1,\ldots,F_k)$ being piecewise on $\Sigma$ with image under $\#$ equal to $(f,g)$ is equivalent to all of the ideal membership requirements
\begin{eqnarray}
\nonumber F_1-f&\in&J_{\rho_1}\\
\nonumber F_2-F_1&\in&J_{\rho_2}\\
\label{equation:piecewise ideal} &\vdots&\\
\nonumber F_k-F_{k-1}&\in&J_{\rho_k}\\
\nonumber g-F_k&\in&J_{\rho_{k+1}}.
\end{eqnarray}

\begin{lemma}\label{lemma:getting-there}
We have
\begin{equation}\label{equation:image sharp}
\mathrm{Im}\bigg( \#:PLP(\sigma_1\cup\cdots\cup\sigma_k)\to 
PLP(\rho_1\cup\rho_{k+1})\bigg)
 = \bigg\{ (f,g)\ \bigg| \ f-g\in J_{\rho_1}+\cdots+J_{\rho_{k+1}}\bigg\}.
\end{equation}
\end{lemma}

\begin{proof}
Adding all of the equations in the list \eqref{equation:piecewise ideal} gives 
$g-f\in J_{\rho_1}+\cdots+J_{\rho_{k+1}}$, and this establishes 
that the left-hand side of \eqref{equation:image sharp} is contained in the right-hand
side. To show the opposite inclusion, suppose $(f,g)$ satisfies
$f-g\in J_{\rho_1}+\cdots+J_{\rho_{k+1}}$
and write
\begin{equation}\label{eq:f-g}
f-g=a_1e(\rho_1)+\cdots+a_{k+1}e(\rho_{k+1})
\end{equation}
for some Laurent polynomials $a_1,\ldots,a_{k+1}\in\Z[\alpha^{\pm1},\beta^{\pm1}]$. Then set
\begin{eqnarray}
\nonumber F_1&=&f-a_1e(\rho_1)\\
\nonumber F_2&=&F_1-a_2e(\rho_2)\\
\label{eq:string} &\vdots&\\
\nonumber F_{k-1}&=&F_{k-2}-a_{k-1}e(\rho_{k-1})\\
\nonumber F_k&=&F_{k-1}-a_ke(\rho_k)=f-(f-g-a_{k+1}e(\rho_{k+1}))=g+a_{k+1}e(\rho_{k+1}).
\end{eqnarray}
By construction of the $F_j$, it follows that $f$, $g$ and the $F_j$ satisfy 
\eqref{equation:piecewise ideal},
and so 
\[
\#\Big( (F_1,\cdots,F_k)\Big) = (f,g),
\]
as desired.
\end{proof}

\begin{remark}\label{remark: representatives}
In the proof of Lemma~\ref{lemma:getting-there}, we could have chosen different representatives
in place of $f$ and $g$, giving rise to a different expression in \eqref{eq:f-g} and so in turn to different Laurent polynomials in \eqref{eq:string}. Nevertheless, these would still satisfy 
\eqref{equation:piecewise ideal} and so the choice of representative does not matter.
\end{remark}

We end this section with a result which is valid for both complete and incomplete fans in $\R^2$. This provides a tool for our analysis of complete fans in $\R^2$ in the next section.

\begin{theorem}\label{thm:2d}
Let $\Sigma$ be a fan in $\R^2$ that admits a proper splitting. Then $K^{2i}_T(X_\Sigma)=PLP(\Sigma)$ for each $i\in \Z$, and the following statements are equivalent.
\begin{enumerate}
\item For each proper splitting, the map $\#$ in \eqref{equation:mv}
is surjective.
\item For some proper splitting, the map $\#$ in \eqref{equation:mv}
is surjective.
\item $K^{2i+1}_T(X_\Sigma) = 0$ for each $i\in \Z$.
\item  $K^*_T(X_\Sigma)\cong P_K(\Sigma)$ as $K^*_T$-algebras.
\end{enumerate}
\end{theorem}

\begin{proof}
Let $\Sigma$ be a fan in $\R^2$ that admits a proper splitting, and choose a proper splitting $\Sigma=\Delta'\cup\Delta''$. Then $\Delta'$ and $\Delta''$ are incomplete fans, so Lemma~\ref{lemma:incomplete} guarantees that both $K^*_T(X_{\Delta'})\cong P_K(\Delta')$ and $K^*_T(X_{\Delta''})\cong P_K(\Delta'')$ as $K^*_T$-algebras. Hence the long exact Mayer-Vietoris sequence 
\eqref{equation:mv} for this splitting reduces to a $4$-term exact sequence
\begin{equation}\label{equation:2dsplit}
0\to K^{2i}_T(X_{\Sigma})\longrightarrow PLP(\Delta')\oplus PLP(\Delta'')\stackrel{\#}
{\longrightarrow} PLP(\Delta'\cap\Delta'')\longrightarrow K^{2i+1}_T(X_{\Sigma})\to 0
\end{equation}
for each $i\in \Z$. Thus we may identify $K^{2i}_T(X_{\Sigma})=\ker(\#)=PLP(\Sigma)$, as required.

We next show that the four given statements are equivalent.

\vskip 0.1in
\noindent $\mathbf{\big(\,(1)\Lra (2)\,\big)}$  This is straightforward. \hfill \ding{52}

\vskip 0.1in
\noindent $\mathbf{\big(\,(2)\Lra (3)\,\big)}$ Fix a proper splitting for which $\#$ is surjective. Then it is straightforward from \eqref{equation:2dsplit} that $K^{2i+1}_T(X_{\Sigma})=0$ for each $i\in\Z$. \hfill \ding{52}

\vskip 0.1in
\noindent $\mathbf{\big(\,(3)\Lra (1)\,\big)}$ Let $\Sigma=\Delta'\cup\Delta''$ be any proper splitting. Then we have a $4$-term exact sequence  \eqref{equation:2dsplit}. But we are assuming $K^{2i+1}_T(X_{\Sigma})=0$, so we must have that $\#$ is surjective.  \hfill \ding{52}

\vskip 0.1in
\noindent $\mathbf{\big(\,(3)\Lra (4)\,\big)}$ When $K^{2i+1}_T(X_\Sigma) = 0$, the $4$-term exact sequence \eqref{equation:2dsplit} becomes a $3$-term exact sequence. Using the identification $K^{2i}_T(X_\Sigma)=PLP(\Sigma)$, this is
\begin{equation}\label{equation:2d-no-odd}
0\to PLP(X_{\Sigma})\longrightarrow PLP(\Delta')\oplus PLP(\Delta'')\stackrel{\#}
{\longrightarrow} PLP(\Delta'\cap\Delta'')\to 0.
\end{equation}
The identifications $K_T^{2i}(X_\Gamma)\cong PLP(\Gamma)$, for $\Gamma=\Delta',\Delta''$, and $\Delta'\cap\Delta''$ induce algebra isomorphisms $K_T^{*}(X_\Gamma)\cong P_K(\Gamma)$.  We may thus assemble the sequences \eqref{equation:2d-no-odd} to achieve an algebra isomorphism $K_T^*(X_\Sigma)\cong P_K(\Sigma)$. \hfill \ding{52}

\vskip 0.1in
\noindent $\mathbf{\big(\,(4)\Lra (3)\,\big)}$ This is straightforward.\hfill\ding{52}
\end{proof}

\section{Complete fans in $\R^2$}\label{section:R^2-complete}

We next consider surjectivity of the map $\#$ of \eqref{equation:mv} in the context of complete fans in $\R^2$. 
The ideal $J_{\rho_1}+\cdots+J_{\rho_{k+1}}$ appeared in Lemma \ref{lemma:getting-there}, and our study 
involves further ideals of this form. Thus we make a small algebraic digression to discuss these \ouremph{lattice ideals}.

The theory of lattice ideals in polynomial rings is well-developed; see  \cite[\S7.1]{ms:cca}, for example. We make the natural 
generalisation to Laurent polynomial rings here. Given a lattice $L\leq\Z^s$ for some $s>0$, we write $J_L$ for the 
\ouremph{Laurent polynomial lattice ideal} of $L$,
\[
J_L=\bigg\langle {\boldsymbol\alpha}^u-{\boldsymbol\alpha}^v\ \bigg|\ u-v\in L, u,v\in\Z^s\bigg\rangle 
< \Z[{\alpha}_1^{\pm 1},\dots,{\alpha}_s^{\pm 1}].
\]
If $\mathbb{L}$ is a matrix whose columns $\ell^1,\ldots,\ell^r$ $\Z$-span the lattice $L$, 
we write 
\[
J_\mathbb{L}=
\bigg\langle 1-{\boldsymbol\alpha}^{\ell^1}\ ,\ \dots\ ,\ 1-{\boldsymbol\alpha}^{\ell^r} \bigg\rangle
< \Z[{\alpha}_1^{\pm 1},\dots,{\alpha}_s^{\pm 1}].
\]
The lattice ideal lemma for Laurent polynomial rings (Corollary~\ref{cor:lattice LPL}) guarantees that $J_L = J_{\mathbb{L}}$. Thus, in the terminology of Section \ref{section:background}, for $\mathbb{L} = [ \nu_1 \cdots \nu_{n-d}]$, the ideal $J_\sigma=J_{\mathbb{L}}$ is the lattice ideal for the sublattice $\sigma^\perp\cap M$. Returning to the context of Lemma~\ref{lemma:getting-there}, the lattice ideal lemma for Laurent polynomial rings (Corollary~\ref{cor:lattice LPL}) guarantees that $J_{\rho_1}+\cdots+J_{\rho_{k+1}}=J_L$, where $L$ is the lattice $\Z$-spanned by the normal vectors to the rays $\rho_1,\dots,\rho_{k+1}$. For an incomplete fan $\Sigma\subset\R^2$ which is a clump, with rays $\rho_1,\dots, \rho_{k+1}$, we denote this lattice ideal by $J_\Sigma = J_L =J_{\rho_1}+\cdots+J_{\rho_{k+1}}$.
We provide further discussion of the lattice ideal lemma in Appendix \ref{section:lil}, contenting ourselves here with an immediate corollary of Proposition~\ref{prop:lattice subset}.

\begin{lemma}\label{lemma:lil-for-2d}
In $\Z[\alpha^{\pm1},\beta^{\pm1}]$, we have 
$J_{\rho_1}+\cdots+J_{\rho_{k+1}}=\langle1-\alpha,1-\beta\rangle$ if and only 
if the primitive generators of the rays
$\rho_1,\ldots,\rho_{k+1}$ span (over $\Z$) the lattice $\Z^2$. \qed
\end{lemma}

We now continue our study of fans in $\R^2$.

\begin{proposition}\label{prop:together}
Let $\Sigma$ be a complete fan in $\R^2$, and consider a splitting $\Sigma=\Delta'\cup\Delta''$ into two clumps, with intersection $\Delta'\cap \Delta''$ equal to two rays $\rho_1\cup\rho_{k+1}$. Then we have
\begin{equation}\label{equation:image full sharp}
\mathrm{Im}\bigg( \#\colon PLP(\Delta')\oplus 
PLP(\Delta'') \to
PLP(\rho_1\cup\rho_{k+1})\bigg)
 = \bigg\{ (f,g)\in PLP(\rho_1\cup\rho_{k+1})\ \bigg| \ f-g\in J_{\Delta'}+J_{\Delta''}\bigg\}.
\end{equation}
\end{proposition}

\begin{proof}
Let $\Sigma=\Delta'\cup\Delta''$ be such a splitting.
Denote the maximal cones of $\Delta'$ by
$\sigma_1,\ldots,\sigma_k$ and the rays by 
$\rho_1,\ldots,\rho_{k+1}$, numbered so that 
$\sigma_i\cap\sigma_{i+1}=\rho_{i+1}$ for $1\leq i\leq k-1$.
Similarly, denote the maximal cones of $\Delta''$ by
$\tau_1,\ldots,\tau_\ell$ and the rays by
$\delta_1,\ldots,\delta_{\ell+1}$, numbered so that 
$\tau_i\cap\tau_{i+1}=\delta_{i+1}$ for $1\leq i\leq \ell-1$, and so
that $\rho_1=\delta_{\ell+1}= \tau_{\ell}\cap \sigma_1$
and $\rho_{k+1}=\delta_{1}= \tau_{1}\cap \sigma_k$.

We first show that the left-hand side of \eqref{equation:image full sharp}
is contained in the right-hand side.  Let
\[
((F_1',\dots,F_k'),(F_1'',\dots,F_\ell''))
\in PLP(\Delta')\oplus PLP(\Delta'').
\]
Then by Lemma~\ref{lemma:getting-there}, we have that
$\#\Big((F_1',\dots,F_k')\Big) = (f',g')$ satisfying $f'-g'\in J_{\Delta'}$, and that
$\#\Big((F_1'',\dots,F_\ell'')\Big) = (f'',g'')$
satisfying $f''-g''\in J_{\Delta''}$.  So we have
\begin{eqnarray*}
\#\Big((F_1',\dots,F_k'),(F_1'',\dots,F_\ell'')\Big) & =  & (f'-g') - 
(f''-g'')\in J_{\Delta'}+J_{\Delta''},
\end{eqnarray*}
as desired.

Next, we show that the right-hand side of \eqref{equation:image full sharp}
is contained in the left-hand side.  Suppose $(f,g)\in PLP(\rho_1\cup\rho_{k+1})$ satisfies $f-g\in J_{\Delta'}+J_{\Delta''}$.
Then we can write
$f-g=A-B$, where $A\in J_{\Delta'}$ and $B\in J_{\Delta''}$.  Setting
\begin{eqnarray*}
f' = (A-B) & \phantom{MOVE} & g' = -B\\
f''=f-(A-B)=g & & g'' = g+B,
\end{eqnarray*}
we compute $f'+f''=f$ and $g'+g''=g$.  Moreover, we have $f'-g' = A\in J_{\Delta'}$ and
also 
$f''-g''=-B\in J_{\Delta''}$.  Then by Lemma~\ref{lemma:getting-there},
we know that there must exist $(F_1',\dots,F_k')$ and $(F_1'',\dots,F_\ell'')$
satisfying $\#\Big((F_1',\dots,F_k')\Big) = (f',g')$ and
$\#\Big((F_1'',\dots,F_\ell'')\Big) = (f'',g'')$.  But then
\begin{eqnarray*}
\#\Bigg(\Big((F_1',\dots,F_k'), -(F_1'',\dots,F_\ell'')\Big)\Bigg) &= &\#\Big((F_1',\dots,F_k')\Big) -  
\#\Big(-(F_1'',\dots,F_\ell'')\Big) \\
& = & (f',g')+ (f'',g'')\\
& = & (f,g),
\end{eqnarray*}
and so $(f,g)\in \mathrm{Im}(\#)$.  This completes the proof.
\end{proof}

We may now prove our main result for fans in $\R^2$.

\begin{theorem}\label{theorem:R^2}
Let $\Sigma$ be a fan in $\R^2$.
\begin{enumerate}
\item\label{item:1}If $\Sigma$ is incomplete, then $K^*_T(X_\Sigma)\cong P_K(\Sigma)$ as a $K^*_T$-algebra.
\item\label{item:2}If $\Sigma$ is complete, then $K^*_T(X_\Sigma)\cong P_K(\Sigma)$ as a $K^*_T$-algebra if and only if the primitive generators of the rays of $\Sigma$ span (over $\Z$) the lattice $\Z^2$.
\end{enumerate}
\end{theorem}

\begin{proof}
For \eqref{item:1} nothing beyond Lemma \ref{lemma:incomplete} is required. We now turn our attention to \eqref{item:2} and take a complete fan $\Sigma$ in $\R^2$. It is immediate that for any such fan we may choose a splitting into two clumps, $\Sigma=\Delta'\cup\Delta''$ with intersection $\Delta'\cap\Delta''$ equal to the union of two rays $\rho_{1}\cup\rho_{k+1}$. By Theorem \ref{thm:2d} it is now necessary and sufficient to show that, for our splitting, the map $\#$ in \eqref{equation:mv} is surjective if and only if the primitive generators of the rays of $\Sigma$ span (over $\Z$) the lattice $\Z^2$.

By Proposition \ref{prop:together}, the image of $\#$ is 
\[
\bigg\{ (f,g)\in PLP(\rho_1\cup\rho_{k+1})\ \bigg| \ f-g\in J_{\Delta'}+J_{\Delta''}\bigg\},
\]
and this is all of $PLP(\rho_1\cup\rho_{k+1})$ if and only if $J_{\Delta'}+J_{\Delta''}=J_{\{0\}}=\langle 1-\alpha,1-\beta\rangle$. By Lemma \ref{lemma:lil-for-2d}, this is the case if and only if the primitive generators of the rays of $\Sigma$ span (over $\Z$) the lattice $\Z^2$, as required.
\end{proof}

\begin{example}[Weighted projective spaces and fake weighted projective spaces]\label{example:fwps}
Amongst the best known examples of toric varieties are \ouremph{weighted projective spaces}; in fact each one is also a \ouremph{fake weighted projective space} (the class of the latter is strictly larger than the class of the former). Fake weighted projective spaces, and their relationship to weighted projective spaces, are discussed in \cite{ak:fwps}.

Let $\Sigma$ be a complete fan in $\R^n$ whose rays have primitive generators $v_0,\ldots,v_n\in\Z^n$ such that $\R^n=\mathrm{Span}_{\R_{\geq0}}(v_0,\ldots,v_n)$; as a consequence one may find coprime $\chi_0,\ldots,\chi_n\in\Z_{>0}$, unique up to order, such that $\chi_0 v_0+\cdots+\chi_n v_n=0$. Then $X_\Sigma$ is a fake weighted projective space with weights $(\chi_0,\ldots,\chi_n)$; if in addition $\Z^n=\mathrm{Span}_{\Z}(v_0,\ldots,v_n)$ then $X_\Sigma$ is also a weighted projective space with weights $(\chi_0,\ldots,\chi_n)$.

From the definitions and using Theorem \ref{theorem:R^2}, it is immediate that if $\Sigma$ is a fan in $\R^2$ such that $X_\Sigma$ is a fake weighted projective space but not a weighted projective space, then $K^*_T(X_\Sigma)$ is {\em not} isomorphic to $P_K(\Sigma)$. We deduce from Theorem \ref{thm:2d} that $K^{\mathrm{odd}}_T(X_\Sigma)\neq0$ and that there is no proper splitting for which the map $\#$ in \eqref{equation:mv} is surjective. On the other hand, if $\Sigma$ is a fan in $\R^2$ such that $X_\Sigma$ is a weighted projective space, then $K^*_T(X_\Sigma)\cong P_K(\Sigma)$ as $K^*_T$-algebras. This latter class includes examples such as $\mathbb{P}(2,3,5)$, which is {\em not} a divisive weighted projective space and hence its equivariant $K$-theory ring is not computed in \cite{HHRW}. 
\end{example}

\begin{example}[Hirzebruch surfaces]\label{example:Hirzebruch}
As discussed in \cite[Example 3.1.16]{colisc:tv}, the Hirzebruch surface $\mathcal{H}_r$ for $r=1,2,3,\ldots$ is the toric variety arising from the complete fan $\Sigma_r$ in $\R^2$ with rays $(1,0)$, $(0,\pm1)$ and $(-1,r)$; in the case of $r=1$, then $\mathcal{H}_r$ is nothing more than the product $\mathbb{P}^1\times\mathbb{P}^1$. Theorem \ref{theorem:R^2} applies and we deduce that $K^*_T(\mathcal{H}_r)\cong P_K(\Sigma_r)$ as a $K^*_T$-algebra, for $r=1,2,3,\ldots$. Of course, $\mathcal{H}_r$ is polytopal and smooth, so the result also follows from \cite{HHRW}. The result may also be deduced from the treatment of smooth toric varieties in the present paper (Section \ref{section:smooth}) without reliance upon the fact that $\mathcal{H}_r$ is polytopal.
\end{example}

\section{A single cone in $\R^n$ and its boundary}\label{section:single-cone}

The preceding sections highlight how Mayer-Vietoris arguments may be applied to toric varieties; crucial to our study is the correspondence between affine pieces and cones, together with an understanding of the map $\#$ of \eqref{equation:4-term}. As we shall see, it is especially useful to continue with the case of a single cone and its interactions with its boundary.

\begin{lemma}\label{lemma:general single cone to boundary}
Let $\sigma\subset \R^n$ be a $d$-dimensional cone with facets $\tau_1,\ldots,\tau_k$, so that $\partial\sigma=\tau_1\cup\cdots\cup\tau_k$. Then we have
\begin{equation}\label{equation:single cone to boundary}
\mathrm{Im}\Big(\#\colon PLP(\sigma)\to PLP(\tau_1\cup\cdots\cup\tau_k)\Big)
= \Big\{ (F_1,\ldots,F_k)\   \Big|\  F_i-F_j\in J_{\tau_i}+J_{\tau_j}     \Big\}.
\end{equation}
\end{lemma}

\begin{remarks}\label{remarks:sharp}
\begin{enumerate}
\item Strictly, one ought to specify a splitting in order to discuss the map $\#$. Since $\sigma$ is a single cone, we cannot chose a proper splitting. Instead, we take in \eqref{equation:4-term} $\Sigma=\sigma$, $\Delta'=\sigma$ and $\Delta''=\partial\sigma$, so that $\Delta'\cap\Delta''=\partial\sigma$. Strictly, $\#$ is then a map $PLP(\sigma)\oplus PLP(\partial\sigma)\to PLP(\partial\sigma)$. In the lemma, we have restricted to the first summand, and retained the name $\#$ by abuse of notation.
\item Note that in this context, the piecewise condition requires that $F_i-F_j\in J_{\tau_i\cap\tau_j}$. We know that $\tau_i\cap\tau_j$ is a face of each of $\tau_i$ and $\tau_j$. Thus we also have that  $(\tau_i\cap\tau_j)^\perp$ contains both $\tau_i^\perp$ and $\tau_j^\perp$.  Proposition~\ref{prop:lattice subset} then says that $J_{\tau_i\cap\tau_j}\geq J_{\tau_i}+J_{\tau_j}$. This guarantees that the right-hand side of \eqref{equation:single cone to boundary} is indeed a subset of $PLP(\tau_1\cup\cdots\cup\tau_k)$.
\end{enumerate}
\end{remarks}

\begin{proof}[Proof of Lemma \ref{lemma:general single cone to boundary}.]
Let $\sigma$ be a  $d$-dimensional cone in $\R^n$.  That is, $\sigma$ 
is contained
in a $d$-dimensional subspace $V$ of $\R^n$.  The rationality of
$\sigma$ means that $V\cap \Z^n$ is a rank $d$ sublattice of $\Z^n$. Thus, we may apply an element of $SL_n(\Z)$ so that the image $\widehat{\sigma}$ of $\sigma$ is a subset of the subspace given by $x_{d+1}=\cdots=x_n=0$. The $SL_n(\Z)$-transformation induces an equivariant isomorphism of toric varieties $X_\sigma\cong X_{\widehat{\sigma}}$ and hence an isomorphism $PLP(\sigma)\cong PLP(\widehat{\sigma})$.  Thus, without loss of generality, we may assume that $\sigma$ is contained in the coordinate subspace spanned by the standard basis vectors $e_1,\ldots,e_d$.

We note that $PLP(\tau_1\cup\cdots\cup\tau_k)$ consists of tuples $(F_1,\ldots,F_k)$
of Laurent polynomials in $\Z[\alpha_1^{\pm 1},\ldots,\alpha_d^{\pm 1}]$
satisfying the piecewise condition, as follows.
Strictly speaking, each $F_i$ is an equivalence class in a quotient of 
$\Z[\alpha_1^{\pm 1},\ldots,\alpha_n^{\pm 1}]$ by an ideal, but following 
Remark~\ref{remark:equivalence reps}, we may work with representatives 
$F_i\in \Z[\alpha_1^{\pm 1},\ldots,\alpha_n^{\pm 1}]$.
For $F_i$ associated to $\tau_i$, then $F_i$ is a representative of an equivalence class in
\[
\frac{\Z[\alpha_1^{\pm 1},\ldots,\alpha_n^{\pm 1}]}{\langle 1-\alpha_{d+1},\ldots,1-\alpha_n,1-
\boldsymbol\alpha^{\nu_i}\rangle},
\]
where $\nu_i$ is a primitive generator of the lattice $M\cap \tau_i^\perp$. Here, $\tau_i^\perp$ is the concatenation $(\tau_i^{\perp_d},0,\ldots,0)$ where $\tau_i^{\perp_d}$ means the one dimensional subspace of $\R^d=\langle e_1,\ldots, e_d\rangle$ orthogonal to $\tau_i$ when the latter is considered in $\R^d$.

In practice, this means that we may view all of the $F_i$, and indeed an 
element $F\in PLP(\sigma)$, as not involving the variables $\alpha_{d+1},\ldots,\alpha_n$, 
by virtue of the fact that they may be replaced by 1 in any representative 
which involves them. Each $F_i$ ($1\leq i\leq k$) is well-defined only up
to a multiple of $(1-\boldsymbol\alpha^{\nu_i})$.  As we shall see, we shall only
require $F_i$ modulo the ideal $J_{\tau_i}$ and thus, choice of representative
of $F_i$ is immaterial to the {\em existence} of a preimage for a tuple $(F_1,\dots,F_k)$. 
Preimages are certainly {\em not} unique.

We now note that $\#(F) = (F,F,\dots,F)$ for any $F\in PLP(\sigma)$. Thus, because
$F-F = 0$, we immediately have 
\begin{equation*}
\mathrm{Im}\Big(\#\colon PLP(\sigma)\to PLP(\tau_1\cup\cdots\cup\tau_k)\Big)
\subseteq \Big\{ (F_1,\ldots,F_k)\   \Big|\  F_i-F_j\in J_{\tau_i}+J_{\tau_j}     \Big\}.
\end{equation*}

To prove the reverse containment, we start with a tuple $(F_1,\dots,F_k)$ and aim to find
$F\in PLP(\sigma)$ so that  $F-F_i=0\in J_{\tau_i}$. Then we shall have $\#(F) =(F,\dots,F)= (F_1,\dots,F_k)$ in the appropriate quotient ring. So take $(F_1,\dots,F_k)$ and write
\[
(F_1,\ldots,F_k)=(F_1,F_1,\ldots,F_1)+(0,F_2-F_1,\ldots,F_k-F_1).
\]
Now, because $(F_1,\ldots,F_k)$ and $(F_1,F_1,\ldots,F_1)$ are piecewise and satisfy
the right-hand side of \eqref{equation:single cone to boundary}, the same things hold
true for $(0,F_2-F_1,\ldots,F_k-F_1)$.  Because $(F_1,\ldots,F_k)$ satisfies
the right-hand side of \eqref{equation:single cone to boundary}, we know that
$F_i-F_1\in J_{\tau_1}+J_{\tau_i}$, so we may write
\[
F_i-F_1=F_{i}^{(1)}(1-\boldsymbol\alpha^{\nu_1}) + G_{i}^{(1)}(1-\boldsymbol\alpha^{\nu_i})
\]
for some Laurent polynomials $F_i^{(1)}$ and $G_i^{(1)}$. Observe that 
\[
F_{i}^{(1)}(1-\boldsymbol\alpha^{\nu_1}) + G_{i}^{(1)}(1-\boldsymbol\alpha^{\nu_i})=F_{i}^{(1)}(1-\boldsymbol\alpha^{\nu_1})
\] 
in $J_{\tau_i}$. Thus
\[
\Big(0,F_2-F_1,\dots,F_k-F_1\Big) = \Big(0,F_{2}^{(1)}(1-\boldsymbol\alpha^{\nu_1}),
\ldots,F_{k}^{(1)}(1-\boldsymbol\alpha^{\nu_1})\Big)\in PLP(\tau_1\cup\cdots
\cup\tau_k),
\]
and so we have
\begin{equation}\label{equation:general recursive}
(F_1,\ldots,F_k)=(F_1,F_1,\ldots,F_1)+(1-\boldsymbol\alpha^{\nu_1})
(0,F_2^{(1)},\ldots,F_k^{(1)}).
\end{equation}

We now observe that $(0,F_2^{(1)},\ldots,F_k^{(1)})$ is piecewise on $\tau_1\cup\cdots\cup\tau_k$. 
This follows from (\ref{equation:general recursive}) and the fact that $1-\boldsymbol\alpha^{\nu_1}$ is non-zero in the ideals $J_{\tau_2},\ldots,J_{\tau_k}$. Hence we may iterate the process. At the next stage,
\begin{eqnarray*}
(F_1,\ldots,F_k)&=&(F_1,F_1,\ldots,F_1)+(1-\boldsymbol\alpha^{\nu_1})
(0,F_2^{(1)},\ldots,F_k^{(1)})\\
&=&(F_1,F_1,\ldots,F_1)+(1-\boldsymbol\alpha^{\nu_1})(0,F_2^{(1)},\ldots,F_2^{(1)})\\
&&+(1-\boldsymbol\alpha^{\nu_1})(0,0,F_3^{(1)}-F_2^{(1)},\ldots,F_k^{(1)}-F_2^{(1)})\\
&=&(F_1,F_1,\ldots,F_1)+(1-\boldsymbol\alpha^{\nu_1})(0,F_2^{(1)},\ldots,F_2^{(1)})\\
&&+(1-\boldsymbol\alpha^{\nu_1})(1-\boldsymbol\alpha^{\nu_2})(0,0,F_3^{(2)},\ldots,F_k^{(2)}),
\end{eqnarray*}
where $(1-\boldsymbol\alpha^{\nu_2})F_i^{(2)}=F_i^{(1)}-F_2^{(1)}$ (modulo $J_{\tau_i}$).
Eventually one finds that the tuple
\begin{eqnarray*}
(F_1,\ldots,F_k)&=&(F_1,F_1,\ldots,F_1)+(1-\boldsymbol\alpha^{\nu_1})(0,F_2^{(1)},\ldots,F_2^{(1)})\\
&&+(1-\boldsymbol\alpha^{\nu_1})(1-\boldsymbol\alpha^{\nu_2})(0,0,F_3^{(2)},\ldots,F_3^{(2)})\\
&&\vdots\\
&&+(1-\boldsymbol\alpha^{\nu_1})\cdots(1-\boldsymbol\alpha^{\nu_{k-1}})(0,\ldots,0,F_k^{(k-1)}),
\end{eqnarray*}
and that the Laurent polynomial
\begin{eqnarray*}
F & := & F_1+(1-\boldsymbol\alpha^{\nu_1})F_2^{(1)}+(1-\boldsymbol\alpha^{\nu_1})
(1-\boldsymbol\alpha^{\nu_2})F_3^{(2)}
+\cdots+(1-\boldsymbol\alpha^{\nu_1})\cdots(1-\boldsymbol\alpha^{\nu_{k-1}})F_k^{(k-1)}\\
& \in & PLP(\sigma)
\end{eqnarray*}
 satisfies
$\#(F) = (F_1,\ldots,F_k)$.  This completes the proof.
\end{proof}

Lemma~\ref{lemma:general single cone to boundary} provides a powerful tool for analysing the image of $\#\colon PLP(\sigma)\to PLP(\partial\sigma)$ when $\sigma$ is a $d$-dimensional cone in $\R^n$. At one extreme, it allows us to deduce in Lemma~\ref{lemma:smooth cone} that $\#$ is surjective when $\sigma$ is smooth. At the other extreme, there exist non-simplicial $\sigma$ with facets $\tau_1$ and $\tau_2$ such that $\tau_1\cap\tau_2$ has codimension strictly less than $2$ inside $\sigma$. Then $J_{\tau_1}+J_{\tau_2}$ has two generators $1-\boldsymbol\alpha^{\nu_1}$ and $1-\boldsymbol\alpha^{\nu_2}$, whereas $J_{\tau_1\cap\tau_2}$ has strictly more than $2$ generators. In these circumstances $PLP(\partial\sigma)$ is strictly larger than the right hand side of \eqref{equation:single cone to boundary} and so, by Lemma~\ref{lemma:general single cone to boundary}, $\#$ is far from surjective.

\section{Smooth toric varieties}\label{section:smooth}

When $\Sigma$ is a smooth, polytopal fan, $K^*_T(X)$ has been identified with $P_K(\Sigma)$ \cite[Corollary 7.2.1]{HHRW}. That proof relied upon the existence of a polytope in order that the symplectic techniques of \cite{HaL} could be applied. However, there do exist smooth fans which are {\em not} polytopal: an example of such a fan in $\R^3$ is provided in \cite[p71]{ful:itv}. In the present section we provide a treatment of smooth fans without reliance on the hypothesis of being polytopal.

Recall that a cone in $\R^n$ is \ouremph{smooth} if its rays form part of a $\Z$-basis of $\Z^n$. A smooth cone of dimension $d$ has $\binom{d}{b}$ faces of dimension $b$ ($1\leq b\leq d$); in particular, it has exactly $d$ rays and  $d$ facets. It is immediate from the definition that every face of a smooth cone is a smooth cone. A fan is smooth if all of its cones are smooth, and every subfan of a smooth fan is smooth.

\begin{lemma}\label{lemma:smooth cone}
Let $\sigma\subset \R^n$ be a smooth $d$-dimensional cone. Then the map
\[
\#\colon PLP(\sigma)\to PLP(\partial\sigma)
\]
is surjective.
\end{lemma}

\begin{proof}
Let $\sigma$ be a smooth $d$-dimensional cone in $\R^n$. As in the proof of Lemma~\ref{lemma:general single cone to boundary}, we may apply an element of $SL_n(\Z)$ so that the image $\widehat{\sigma}$ of $\sigma$ is a subset of the positive orthant in $\R^n$. However, in this case, smoothness guarantees we can arrange that $\widehat{\sigma}$ is the \ouremph{standard coordinate
subspace cone} spanned by the standard basis vectors $e_1,\ldots,e_d$. The $SL_n(\Z)$-transformation induces an equivariant isomorphism of toric varieties $X_\sigma\cong X_{\widehat{\sigma}}$ and hence an isomorphism $P_K(\sigma)\cong P_K(\widehat{\sigma})$.  Thus, without loss of generality, we may
assume that $\sigma$ is a standard coordinate subspace cone with rays $e_1,\ldots,e_d$.

Let the facets of $\sigma$ be $\tau_i =\sigma\cap \{ x_i = 0\}$ for $i=1,\dots,d$. Thus $J_{\tau_i}=\langle 1-\alpha_i\rangle$ for each $i$, and for $i\neq j$, $J_{\tau_i}+J_{\tau_j}=\langle 1-\alpha_i,1-\alpha_j\rangle=J_{\tau_i\cap\tau_j}$, where the second equality follows because we are working with a standard coordinate subspace cone. Lemma~\ref{lemma:general single cone to boundary} now guarantees that
\[
\mathrm{Im}\Big(\#\colon PLP(\sigma)\to PLP(\tau_1\cup\cdots\cup\tau_d)\Big)
= \Big\{ (F_1,\ldots,F_d)\   \Big|\  F_i-F_j\in J_{\tau_i\cap \tau_j}\Big\}
= PLP(\tau_1\cup\cdots\cup\tau_d).
\]
Hence $\#$ is surjective, as required.
\end{proof}

Given $F\in PLP(\partial\sigma)$, Lemma \ref{lemma:smooth cone} guarantees that there is $\widetilde{F}\in PLP(\sigma)$ with $\#(\widetilde{F})=F$; we say that $\widetilde{F}$ \ouremph{extends} $F$.

\begin{lemma}\label{lemma:smooth cone subfan}
Let $\sigma\subset \R^n$ be a smooth $d$-dimensional cone, and $\Gamma\subseteq\sigma$ a non-empty subfan. Then the map
\[
\#\colon PLP(\sigma)\to PLP(\Gamma)
\]
is surjective.
\end{lemma}

\begin{proof}
Let $\Gamma$ be a non-empty subfan of the smooth $d$-dimensional cone $\sigma$ in $\R^n$. We will work with representatives of equivalence classes on each cone of $\Gamma$ or $\sigma$. In particular, we view an element $F\in PLP(\Gamma)$ as a collection of compatible Laurent polynomials $F=(F_\gamma)_{\gamma\in\Gamma}$.

Suppose we are given $F\in PLP(\Gamma)$. Write $\sigma(k)$ for the collection of $k$-dimensional cones of $\sigma$ ($0\leq k\leq d$), so that $\sigma(d)=\{\sigma\}$. We construct $\widetilde{F}\in PLP(\sigma)$ with $\#(\widetilde{F})=F$ as follows. First $\{0\}$ is always a sub-cone of $\Gamma$ and so we are given $F_{\{0\}}$. For a ray $\rho\in\Gamma$, we are given $F_{\rho}$. If $\rho$ is a ray of $\sigma$ but not $\Gamma$ we may use Lemma~\ref{lemma:smooth cone} to extend $F_{\{0\}}$ to $F_\rho$. In this way we obtain an element of $PLP(\Gamma\cup\sigma(1))$. Now, for a $2$-dimensional cone $\upsilon\in\sigma$, if $\upsilon\in\Gamma$ we are given $F_\upsilon$. If $\upsilon\not\in\Gamma$ we have $F_{\partial\upsilon}$ and hence may appeal to Lemma~\ref{lemma:smooth cone} to obtain $F_\upsilon$. In this way we obtain an element of $PLP(\Gamma\cup\sigma(2))$. We continue this process recursively until we obtain an element $\widetilde{F}\in PLP(\Gamma\cup\sigma(d))=PLP(\sigma)$. By construction, $\#(\widetilde{F})=F$.
\end{proof}

\begin{remark}
Our construction of $\widetilde{F}$ relied upon choices of representatives of equivalence classes. Different choices may yield a different $\widetilde{F}$, but it is clear that they are immaterial to the {\em existence} of $\widetilde{F}$.
\end{remark}

\begin{lemma}\label{lemma:k theory boundary cone PLP}
Let $\sigma\subset \R^n$ be a smooth $d$-dimensional cone ($d\geq1$) in $\R^n$ and for each $k$ with $1\leq k\leq d$ let $\Delta^d_k$ be a union of $k$ facets of $\sigma$. Then
\[
K^*_T(X_{\Delta^d_k})\cong P_K(\Delta^d_k)\mbox{ as $K^*_T$-algebras.}
\]
In particular, $K^*_T(X_{\partial\sigma})\cong P_K(\partial\sigma)$ as $K^*_T$-algebras.
\end{lemma}

\begin{proof}

For each $d$ and $k$ with $1\leq k\leq d$ we say $(d,k)$ holds if and only if $K^*_T(X_{\Delta^d_k})\cong P_K(\Delta^d_k)$ as $K^*_T$-algebras. For each $d\geq1$ we say that $(d,\ast)$ holds if and only if $(d,k)$ holds for all $1\leq k\leq d$. We will prove that $(d,\ast)$ holds for all $d\geq 1$; this will prove the lemma. We proceed by induction on $d$.

For the initial step, we must show that $(1,\ast)$ holds, but this is no more than showing that $(1,1)$ holds. But $\Delta^1_1=\{0\}$ and so $(1,1)$ holds by \eqref{equation:one-cone}. This completes the initial step.

For the inductive step, suppose $(1,\ast),(2,\ast),\ldots,(r-1,\ast)$ hold for some $1\leq r\leq d$. We shall show that $(r,\ast)$ holds. This will complete the inductive step, and the proof.

We show that $(r,\ast)$ holds by an inductive argument of its own. By \eqref{equation:one-cone}, $(r,1)$ holds, so now suppose that $(r,1),(r,2),\ldots,(r,s-1)$ hold for some $1\leq s\leq r$. All that remains is to deduce that $(r,s)$ holds. So let $\sigma\subset\R^n$ be a smooth cone of dimension $r$ and let $\Delta^r_s$ be a union of $s$ facets of $\sigma$; write $\Delta^r_s=\tau_1\cup\cdots\cup\tau_s$ for facets $\tau_1,\ldots,\tau_s$ of $\sigma$. Set $\Delta'=\tau_1\cup\cdots\cup\tau_{s-1}$ and $\Delta''=\tau_s$, so $\Delta'\cup\Delta''=\Delta^r_s$ and
\[
\Delta'\cap \Delta'' = \left(\tau_1\cup\cdots\cup\tau_{s-1}\right) \cap \tau_s
=  (\tau_1\cap\tau_s)\cup\cdots\cup(\tau_{s-1}\cap\tau_s).
\]
We note that each $\tau_i\cap\tau_s$ ($1\leq i\leq s-1$) is a facet of the smooth $(r-1)$-cone $\tau_s\subset\R^n$, so $\Delta'\cap\Delta''$ is of the form $\Delta^{r-1}_{s-1}$. Thus we have $K^*_T(X_{\Gamma})\cong P_K(\Gamma)$ as $K^*_T$-algebras for 
$\Gamma =\Delta',\Delta''$ and $\Delta'\cap\Delta''$ because $(r,s-1)$ holds, by \eqref{equation:one-cone} and because $(r-1,\ast)$ holds, respectively. This means that the Mayer-Vietoris sequence \eqref{equation:mv} splits into a $4$-term sequence
\[
0\to K^{2i}_T(X_{\Delta^r_s})\longrightarrow PLP(\Delta')\oplus PLP(\Delta'')\stackrel{\#}
{\longrightarrow} PLP(\Delta'\cap\Delta'')\longrightarrow K^{2i+1}_T(X_{\Delta^r_s})\to 0.
\]
Since $\Delta'\cap\Delta''$ is a non-empty subfan of the smooth cone $\tau_s$, we may apply Lemma~\ref{lemma:smooth cone subfan} to see that $\#$ is surjective, from the second summand. It follows that 
\[
K^{2i}_T(X_{\Delta^r_s})\cong PLP(\Delta^r_s) \mbox{ and } K^{2i+1}_T(X_{\Delta^r_s})= 0
\]
for each $i\in\Z$. Since all the identifications made in the inductive step were as $K^*_T$-algebras, we may assemble the $4$-term sequences to achieve an algebra isomorphism $K^*_T(X_{\Delta^r_s}) \cong P_K(\Delta^r_s)$,
as desired.
\end{proof}

We are now ready to deduce our main result for smooth toric varieties.

\begin{theorem}\label{theorem:smooth fans}
If $\Sigma$ is a smooth fan in $\R^n$ then $K^*_T(X_\Sigma)\cong P_K(\Sigma)$ as $K^*_T$-algebras.
\end{theorem}

\begin{proof}
Let $\Sigma\subseteq \R^n$ be a smooth fan.  Enumerate all the cones of $\Sigma$ as $\sigma_0,\sigma_1,\sigma_2,\dots,\sigma_N$ in an order so that $\dim(\sigma_i)\leq \dim(\sigma_j)$ whenever $i<j$.  Let $\displaystyle{\Sigma_k = \bigcup_{i=0}^k \sigma_i}$.  We will prove that $K_T^*(\Sigma_k) \cong P_K(\Sigma_k)$ as $K^*_T$-algebras for $k\geq 0$ by induction on $k$.

The base case is when $k=0$. We know that $K^*_T(X_{\sigma_0})\cong P_K(\sigma_0)$ as $K^*_T$-algebras immediately from \eqref{equation:one-cone}.

Assume inductively that the statement holds for $\Sigma_{k-1}$ (some $k\leq N$) and consider $\Sigma_k=\sigma_k\cup\Sigma_{k-1}$. All proper faces of $\sigma_k$ must be in $\Sigma_{k-1}$, because $\dim(\sigma_i)\leq \dim(\sigma_j)$ for all $i<j$. It follows that $\sigma_k\cap\Sigma_{k-1} = \partial\sigma_k$. We now consider the Mayer-Vietoris sequence \eqref{equation:mv} with $\Delta'=\sigma_k$ and $\Delta''=\Sigma_{k-1}$. We have $K^*_T(X_\Gamma)\cong P_K(\Gamma)$ as $K^*_T$-algebras for $\Gamma=\Delta',\Delta''$ and $\Delta'\cap\Delta''$ by \eqref{equation:one-cone}, the inductive hypothesis and Lemma~\ref{lemma:k theory boundary cone PLP}, respectively. This means that the Mayer-Vietoris sequence \eqref{equation:mv} splits into a $4$-term sequence
\[
0\to K^{2i}_T(X_{\Sigma_k})\longrightarrow PLP(\Delta')\oplus PLP(\Delta'')\stackrel{\#}
{\longrightarrow} PLP(\Delta'\cap\Delta'')\longrightarrow K^{2i+1}_T(X_{\Sigma_k})\to 0.
\]
We now apply Lemma~\ref{lemma:smooth cone subfan} to see that $\#$ is surjective, from the first summand. It follows that 
\[
K^{2i}_T(X_{\Sigma_k})\cong PLP(\Sigma_k) \mbox{ and } K^{2i+1}_T(X_{\Sigma_k})= 0
\]
for each $i\in\Z$. Since all the identifications made in the inductive step were as $K^*_T$-algebras, we may assemble the $4$-term sequences to achieve an algebra isomorphism $K^*_T(X_{\Sigma_k}) \cong P_K(\Sigma_k)$, as desired.
\end{proof}

Our final result in this section concerns the map $\#$ in the case of a smooth fan $\Sigma$ and any non-empty subfan $\Gamma$.

\begin{theorem}\label{theorem:sharp onto subfan of smooth fan}
Let $\Sigma$ be a smooth fan in $\R^n$ and $\Gamma\subseteq\Sigma$ a non-empty subfan. Then the map
\[
\#\colon PLP(\Sigma)\to PLP(\Gamma)
\]
is surjective.
\end{theorem}

\begin{remark}
In \eqref{equation:4-term}, take $\Delta'=\Sigma$ and $\Delta''=\Gamma$, so that $\Delta'\cap\Delta''=\Gamma$. Strictly, $\#$ is then a map $PLP(\Sigma)\oplus PLP(\Gamma)\to PLP(\Gamma)$, but in the spirit of Remarks \ref{remarks:sharp} we have, in the theorem, restricted to the first summand, and retained the name $\#$ by abuse of notation.
\end{remark}

\begin{proof}[Proof of Theorem \ref{theorem:sharp onto subfan of smooth fan}]
Let $\Sigma\subseteq \R^n$ be a smooth fan.  Enumerate the maximal cones of $\Sigma$ as $\sigma_1,\sigma_2,\dots,\sigma_N$ such that if $\sigma_i\in\Gamma$ for some $i$, then $\sigma_j\in\Gamma$ for all $1\leq j\leq i$. Let $F\in PLP(\Gamma)$. We seek $\widetilde{F}\in PLP(\Sigma)$ such that $\#(\widetilde{F})=F$. As usual, we work with representatives of equivalence classes. Different choices may yield a different $\widetilde{F}$ but are immaterial to the {\em existence} of such an $\widetilde{F}$.

If $\Gamma=\Sigma$ there is nothing to prove, so let $\sigma_r$ be such that $\sigma_i\in\Gamma$ for $1\leq i\leq r-1$,  but $\sigma_r\not\in\Gamma$. Each element of $PLP(\Sigma)$ is determined by its values on the maximal cones of $\Sigma$, so we must extend $F$ to each of $\sigma_{r+1},\ldots,\sigma_N$ in a compatible way.

We will construct $\widetilde{F}_i\in PLP(\Gamma\cup\{\sigma_{r+1},\dots,\sigma_{r+i}\})$ for $i=1,\ldots,N-r$ in an inductive manner. Suppose we have constructed $\widetilde{F}_{i-1}$ (in the case $i=1$, we take $\widetilde{F}_0=F$ as our construction). Now $\partial\sigma_{r+i}\cap(\Gamma\cup\{\sigma_{r+1},\ldots,\sigma_{r+i-1}\})$ is a subfan of $\partial\sigma_{r+i}$, non-empty since both contain $\{0\}$. Hence we may apply Lemma~\ref{lemma:smooth cone subfan} to extend $\widetilde{F}_{i-1}|_{\partial\sigma_{r+i}\cap(\Gamma\cup\{\sigma_{r+1},\ldots,\sigma_{r+i-1}\})}$ to $\sigma_{r+i}$. We define $\widetilde{F_i}$ to have the same value as the extension on $\sigma_{r+i}$ and the same value as $\widetilde{F}_{i-1}$ on $\Gamma\cup\{\sigma_{r+1},\ldots,\sigma_{r+i-1}\}$.

The end result $\widetilde{F}=\widetilde{F}_N$ is an element of $\Gamma\cup\{\sigma_1,\ldots,\sigma_N\}=\Sigma$, and by construction, $\#(\widetilde{F})=F$. This completes the proof.
\end{proof}

\begin{remark}
The analysis in this section relied upon $\Sigma$ being smooth. This assumption is necessary to ensure the existence of an element of $SL_n(\Z)$ so that the image of a given smooth cone is a standard coordinate subspace cone (see the proof of Lemma~\ref{lemma:smooth cone}).
\end{remark}

\section{Fans in $\R^n$ with distant or isolated singular cones}\label{section:distant}

Having dealt with smooth fans in Section \ref{section:smooth}, we turn our attention now to fans in $\R^n$ with singularities. An arbitrary such fan is a step too far: instead, we impose control over the singularities we work with.

\begin{definitions}
Let $\Sigma$ be a singular fan in $\R^n$ and let $\sigma\in\Sigma$. We say that $\sigma$ is an \ouremph{isolated singular cone} in $\Sigma$ if $\sigma$ is a singular cone and $\sigma\cap\tau$ is smooth for every cone $\tau\in\Sigma$ with $\tau\neq\sigma$. We say that $\sigma$ is a \ouremph{distant singular cone} if $\sigma$ is a singular cone and $\sigma\cap\tau=\{0\}$ for every singular cone $\tau\in\Sigma$ with $\tau\neq\sigma$. We say that $\Sigma$ is a \ouremph{fan with  isolated singular cones} if every singular cone of $\Sigma$ is an isolated singular cone, and we say that $\Sigma$ is a \ouremph{fan with  distant singular cones} if every singular cone of $\Sigma$ is a distant singular cone.
\end{definitions}

It is immediate that any face of an isolated or distant singular cone is smooth, and in particular isolated and distant singular cones in $\Sigma$ are maximal in $\Sigma$. Isolated and distant singular cones may or may not be simplicial. 

When the fan $\Sigma$ is complete, these notions have easy interpretations in terms of the corresponding toric varieties: isolated singular cones correspond to isolated singular points in the variety. A distant singular cone corresponds to a singular point in the variety which does not lie on the same proper $T$-invariant subvariety as any other singular point in the variety.

We shall see some examples shortly, but first we state our main result about fans with distant singular cones.

\begin{theorem}\label{thm:distant}
Let $\Sigma$ be a fan in $\R^n$ with distant singular cones. Then $K^*_T(X_\Sigma)\cong P_K(\Sigma)$ as $K^*_T$-algebras.
\end{theorem}

\begin{proof}
Let $\Sigma$ be a fan in $\R^n$ with distant singular cones $\sigma_1,\ldots,\sigma_k$. In the Mayer Vietoris sequence \eqref{equation:mv} take $\Delta'=\Sigma\setminus\{\stackrel{\circ}{\sigma_1},\ldots,\stackrel{\circ}{\sigma_k}\}$ and $\Delta''=\sigma_1\cup\cdots\cup\sigma_k$ so that $\Delta'\cap\Delta''=\partial\sigma_1\cup\cdots\cup\partial\sigma_k$. In particular, $\Delta'$ is a smooth fan, and $\Delta'\cap\Delta''$ is a subfan of $\Delta'$. An easy induction argument based on \eqref{equation:one-cone}, Mayer-Vietoris and the fact that the $\sigma_i$ are distant allows one to deduce that $K^*_T(X_{\sigma_1\cup\cdots\cup\sigma_k})\cong P_K(\sigma_1\cup\cdots\cup\sigma_k)$ as $K^*_T$-algebras. This, and Theorem \ref{theorem:smooth fans} allows us to conclude that the Mayer Vietoris sequence reduces to a 4-term exact sequence as in \eqref{equation:4-term},
\[
0\to K^{2i}_T(X_{\Sigma})\longrightarrow PLP(\Delta')\oplus PLP(\Delta'')\stackrel{\#}
{\longrightarrow} PLP(\Delta'\cap\Delta'')\longrightarrow K^{2i+1}_T(X_{\Sigma})\to 0.
\]
We now apply Theorem \ref{theorem:sharp onto subfan of smooth fan} to see that $\#$ is surjective (from the first summand), from which it follows that 
\[
K^{2i}_T(X_{\Sigma})\cong PLP(\Sigma) \mbox{ and } K^{2i+1}_T(X_{\Sigma})=0.
\]
Since all the identifications made in the inductive step were as $K^*_T$-algebras, we may assemble the $4$-term sequences to achieve an algebra isomorphism $K^*_T(X_\Sigma) \cong P_K(\Sigma)$, as desired.
\end{proof}

One might hope that a version of Theorem \ref{thm:distant} applies for fans with isolated singular cones. However, in that case, the \lq easy induction\rq\ mentioned in the proof of Theorem \ref{thm:distant} fails, since we have no way to analyze surjectivity of the map $\#$ in that set-up.

We conclude by providing examples of fans for which Theorem \ref{thm:distant} applies. It is trivial to construct non-complete fans with distant singularities (by taking a collection of singular cones, disjoint other than at $\{0\}$) so we restrict attention to examples for which the fan is complete.

\begin{example}\label{example:square based pyramid}
Let $\Sigma'$ be any complete, smooth fan in $\R^n$ with rays $\rho'_1,\ldots,\rho'_k$ for some $k\geq n+1$. Construct the complete fan $\Sigma$ in $\R^{n+1}$ as follows. To the primitive generator of each $\rho_i'$, adjoin a $1$ to give a primitive element of $\Z^{n+1}$; this specifies rays $\rho_1,\ldots,\rho_k$ in $\R^{n+1}$. Let $\rho$ be the ray in $\R^{n+1}$ with primitive generator $(0,\ldots,0,-1)$. Then $\Sigma$ is the fan with rays $\rho_1,\ldots,\rho_k,\rho$, and such that a collection $\mathcal{C}$ of rays generates a maximal cone of $\Sigma$ if and only if
\[
\mathcal{C}=\{\rho,\rho_{i_1},\ldots,\rho_{i_n}~|~\rho'_{i_1},\ldots,\rho'_{i_n}\mbox{ generate a maximal cone of }\Sigma'\}\mbox{ or }
\mathcal{C}=\{\rho_1,\ldots,\rho_k\}.
\]
Each maximal cone involving $\rho$ is smooth; this follows from the two facts that its projection to the hyperplane $x_{n+1}=0$ is a maximal cone of the smooth fan $\Sigma'$, and that it contains the ray $\rho$. Now consider the maximal cone $\langle\rho_1,\ldots,\rho_k\rangle$. If $k=n+1$, it follows that $\Sigma'$ was the fan of the unweighted projective space $\C P^{n}$. In these circumstances, it is straightforward to check, via \cite[Proposition 2.1]{ak:fwps}, that $\Sigma$ is the fan for the weighted projective space $\mathbb{P}(1,\ldots,1,n+1)$; this is a divisive example, so of little new interest. However, if $k>n+1$ the maximal cone is not simplicial and is a singular cone which is both distant and isolated. Thus Theorem \ref{thm:distant} applies and we may conclude that $K^*_T(X_\Sigma)\cong P_K(\Sigma)$ as a $K^*_T$-algebra.

If we choose $\Sigma'$ to have at least $n+2$ rays, we ensure that the maximal cone $\langle\rho_1,\ldots,\rho_k\rangle$ of $\Sigma$ is not smooth. It follows that there are examples of fans for which Theorem \ref{thm:distant} applies in all dimensions greater than $1$.

As an explicit example of this construction, let $\Sigma'=\Sigma_1$, the fan in $\R^2$ corresponding to the Hirzebruch surface $\mathcal{H}_1\cong\mathbb{P}^1\times\mathbb{P}^1$ (see Example \ref{example:Hirzebruch}). The fan $\Sigma$ in $\R^3$ is shown in Figure \ref{figure:distant-sings}. In more detail, $\Sigma$ has five rays with primitive generators
\[
\rho_1=\left( \begin{array}{r}1\\0\\1\end{array}\right), \rho_2=\left( \begin{array}{r}0\\1\\1\end{array}\right), \rho_3=\left( \begin{array}{r}-1\\0\\1\end{array}\right), \rho_4=\left( \begin{array}{r}0\\-1\\1\end{array}\right), \rho_5=\left( \begin{array}{r}0\\0\\-1\end{array}\right),
\]
five maximal cones
\[
\langle\rho_1,\rho_2,\rho_3,\rho_4\rangle, \langle\rho_1,\rho_2,\rho_5\rangle, \langle\rho_1,\rho_4,\rho_5\rangle, \langle\rho_2,\rho_3,\rho_5\rangle, \langle\rho_3,\rho_4,\rho_5\rangle,
\]
and is the normal fan to a square-based pyramid. Our analysis guarantees that $K^*_T(X_\Sigma)\cong P_K(\Sigma)$ as a $K^*_T$-algebra.
\end{example}

In the absence of the construction of $\Sigma$ from $\Sigma'$, one could check by hand that the fan $\Sigma$ in Figure \ref{figure:distant-sings} is a singular, complete, non-simplical, polytopal fan which has distant singular cones. However, in such circumstances, we prefer to make use of the packages {\tt Polyhedra} and {\tt NormalToricVarieties} for the computer algebra system {\tt Macaulay2}. We supply code for the fan $\Sigma$ shown in Figure \ref{figure:distant-sings} in Appendix \ref{M2 code}. Code for the other explicit examples in this section is very similar, and is available from either author upon request.

\begin{figure}[htb]
\includegraphics[height=2in]{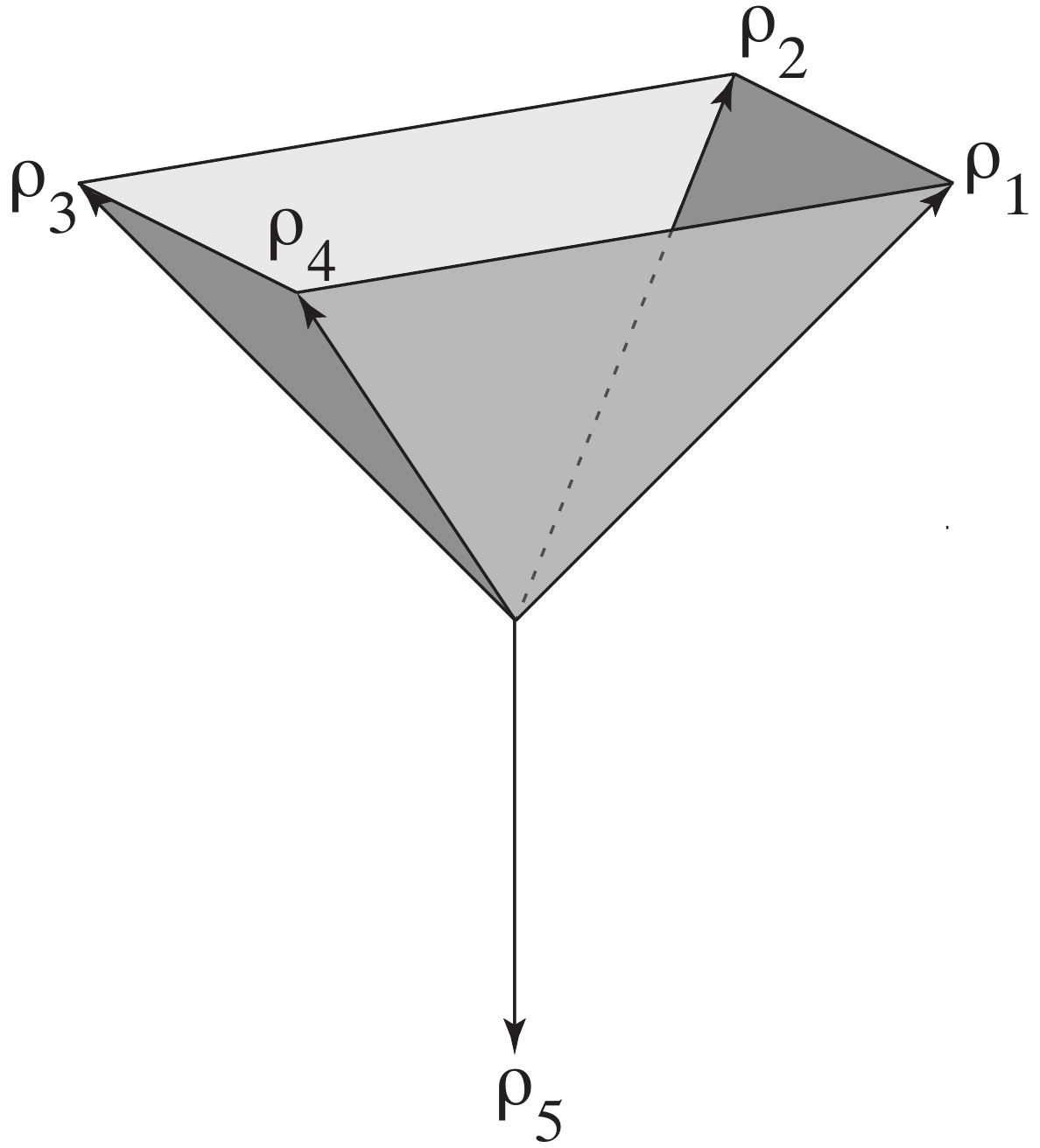}\hspace*{12em}\includegraphics[height=2in]{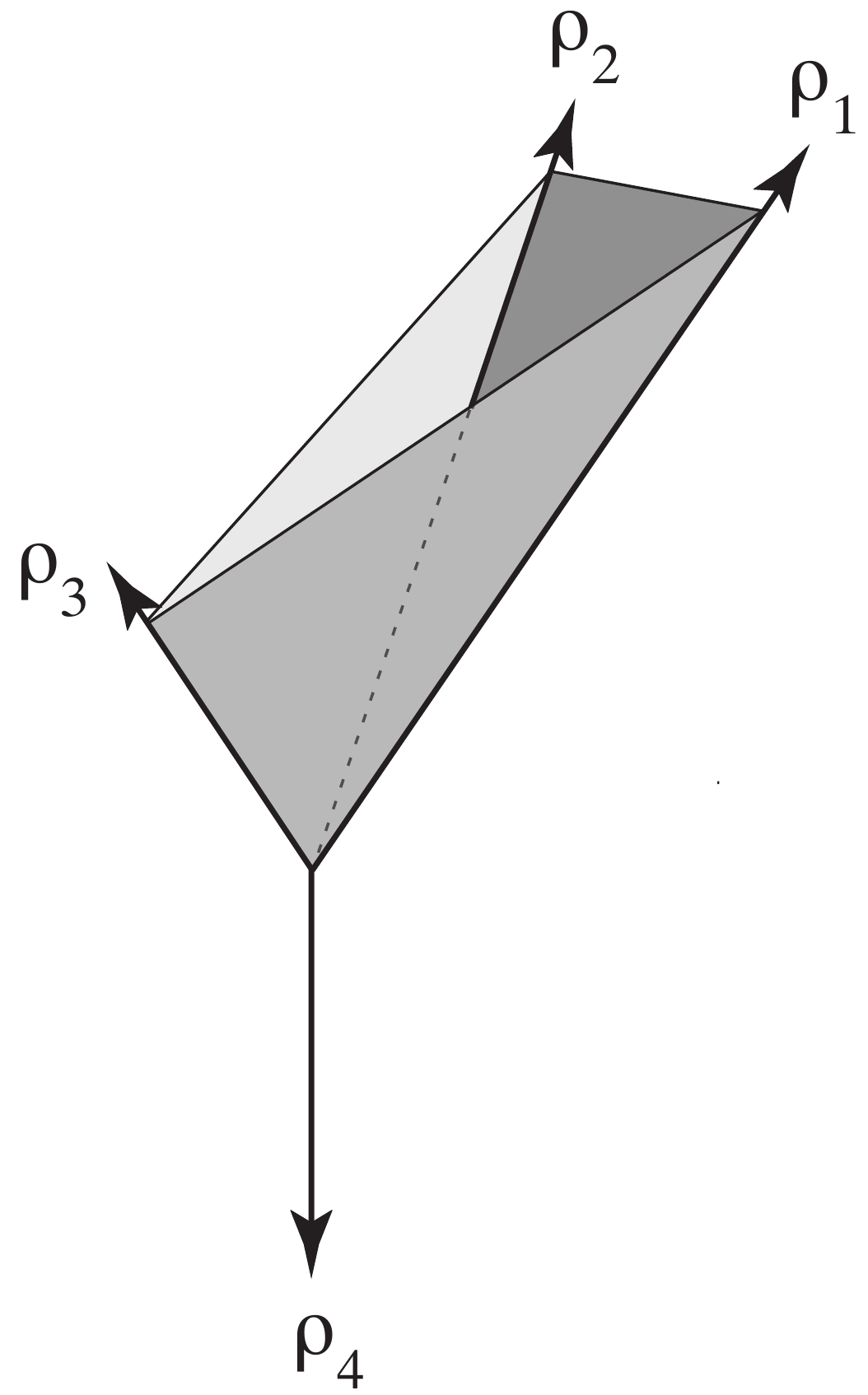}
\caption{Fans $\Sigma$ (left) and $\Delta$ (right) in $\R^3$ as discussed in Examples \ref{example:square based pyramid} and \ref{example:simplicial-distant}. In each case only one maximal cone is highlighted by shading.\label{figure:distant-sings}}
\end{figure}

\begin{example} \label{example:simplicial-distant}
Not all fans with a distant singular cone arise from the construction of Example \ref{example:square based pyramid}. Consider the fan $\Delta$ in $\R^3$ shown in Figure \ref{figure:distant-sings}. In more detail, $\Delta$ has four rays with primitive generators
\[
\rho_1=\left( \begin{array}{r}1\\0\\2\end{array}\right), \rho_2=\left( \begin{array}{r}0\\1\\2\end{array}\right), \rho_3=\left( \begin{array}{r}-1\\-1\\1\end{array}\right), \rho_4=\left( \begin{array}{r}0\\0\\-1\end{array}\right),
\]
and four maximal cones
\[
\langle\rho_1,\rho_2,\rho_3\rangle, \langle\rho_1,\rho_2,\rho_4\rangle, \langle\rho_1,\rho_3,\rho_4\rangle, \langle\rho_2,\rho_3,\rho_4\rangle.
\]
One checks that $\Delta$ is a singular, complete, simplical, polytopal fan which has distant singular cones. Thus Theorem \ref{thm:distant} applies, and we deduce that $K^*_T(X_\Delta)\cong P_K(\Delta)$ as $K^*_T$-algebras.
\end{example}

\begin{example}\label{example:two-isolated-singular-cones}
Consider the fan $\Sigma$ in $\R^3$ with twelve rays whose primitive generators are
\[
\rho_1=\left( \begin{array}{r}  1\\0\\1\end{array}\right),
\rho_2=\left( \begin{array}{r} 0\\1\\1\end{array}\right),
\rho_3=\left( \begin{array}{r} -1\\0\\1\end{array}\right),
\rho_4=\left( \begin{array}{r} 0\\-1\\1\end{array}\right),
\rho_5=\left( \begin{array}{r} 1\\0\\-1\end{array}\right),
\rho_6= \left( \begin{array}{r} 0\\1\\-1\end{array}\right),
\]
\[
\rho_7= \left( \begin{array}{r} -1\\0\\-1\end{array}\right),
\rho_8= \left( \begin{array}{r} 0\\-1\\-1\end{array}\right),
\rho_9= \left( \begin{array}{r} 1\\0\\0\end{array}\right),
\rho_{10}= \left( \begin{array}{r} 0\\1\\0\end{array}\right),
\rho_{11}= \left( \begin{array}{r} -1\\0\\0\end{array}\right),
\rho_{12}= \left( \begin{array}{r} 0\\-1\\0\end{array}\right),
\]
and eighteen maximal cones
\[
\sigma_1 = \langle \rho_1,\rho_2,\rho_3,\rho_4\rangle, 
\sigma_2 = \langle \rho_5,\rho_6,\rho_7,\rho_8\rangle, 
\sigma_3 = \langle \rho_1,\rho_2,\rho_9\rangle, 
\sigma_4 = \langle \rho_2,\rho_9,\rho_{10}\rangle, 
\sigma_5 = \langle \rho_2,\rho_3,\rho_{10}\rangle,
\]
\[
\sigma_6 = \langle \rho_3,\rho_{10},\rho_{11}\rangle,
\sigma_7 = \langle \rho_3,\rho_4,\rho_{11}\rangle, 
\sigma_8 = \langle \rho_4,\rho_{11},\rho_{12}\rangle, 
\sigma_9 = \langle \rho_1,\rho_4,\rho_{12}\rangle, 
\sigma_{10} = \langle \rho_1,\rho_9,\rho_{12}\rangle, 
\]
\[
\sigma_{11} = \langle \rho_5,\rho_6,\rho_9\rangle, 
\sigma_{12} = \langle \rho_6,\rho_9,\rho_{10}\rangle,
\sigma_{13} = \langle \rho_6,\rho_7,\rho_{10}\rangle, 
\sigma_{14} = \langle \rho_7,\rho_{10},\rho_{11}\rangle, 
\sigma_{15} = \langle \rho_7,\rho_8,\rho_{11}\rangle, 
\]
\[
\sigma_{16} = \langle \rho_8,\rho_{11},\rho_{12}\rangle, 
\sigma_{17} = \langle \rho_5,\rho_8,\rho_{12}\rangle,
\sigma_{18} = \langle \rho_5,\rho_9,\rho_{12}\rangle.
\]
\noindent The pictures in Figure \ref{figure:2-distant-singular-cones} may help in visualising $\Sigma$. One checks that $\Sigma$ is a singular, complete, non-simplical, non-polytopal fan with distant singular cones. Thus Theorem \ref{thm:distant} applies, and we deduce that $K^*_T(X_\Sigma)\cong P_K(\Sigma)$ as $K^*_T$-algebras.
\end{example}

\begin{figure}[htb]
\includegraphics[height=2in]{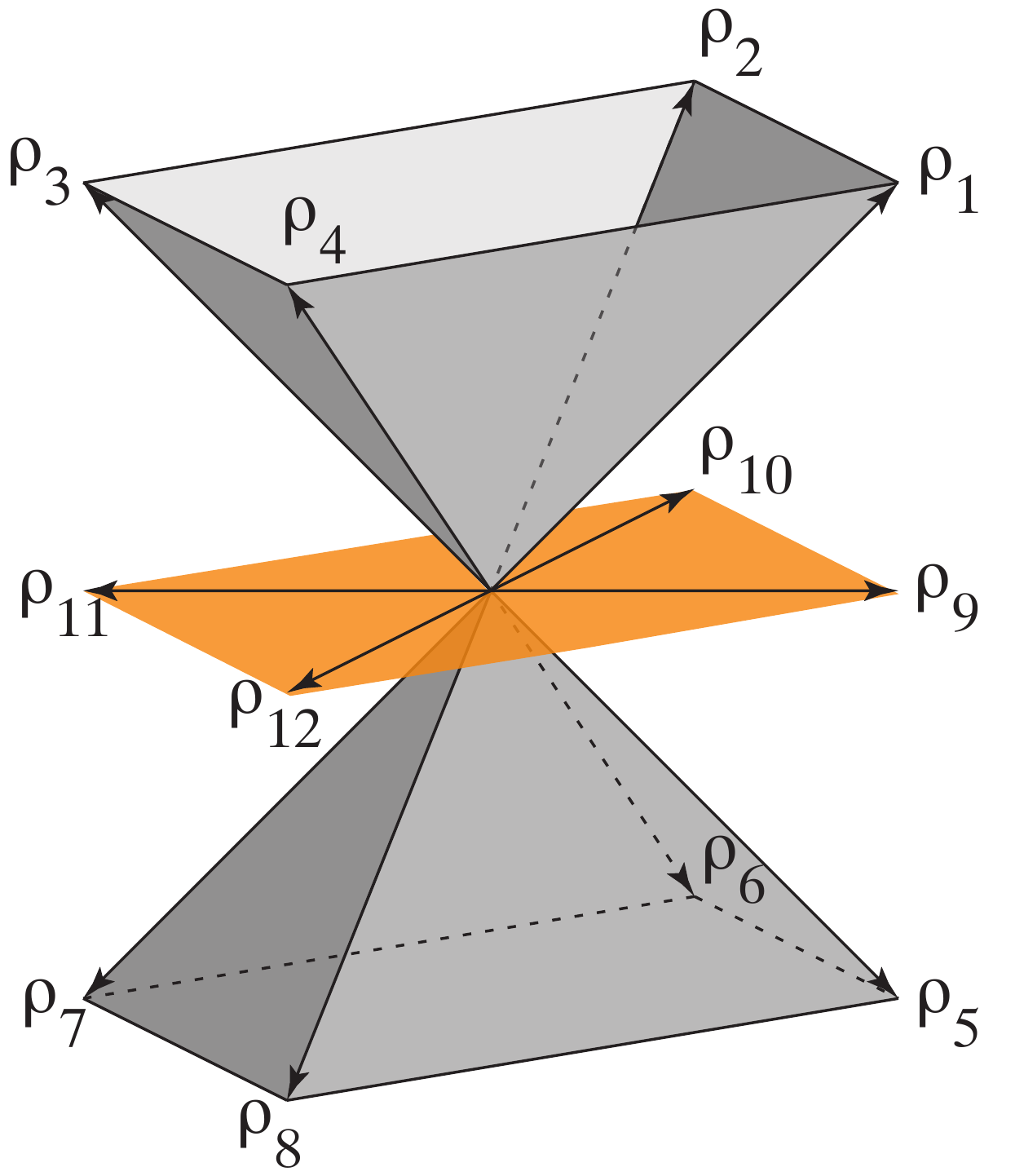} \hskip 0.5in \includegraphics[height=2in]{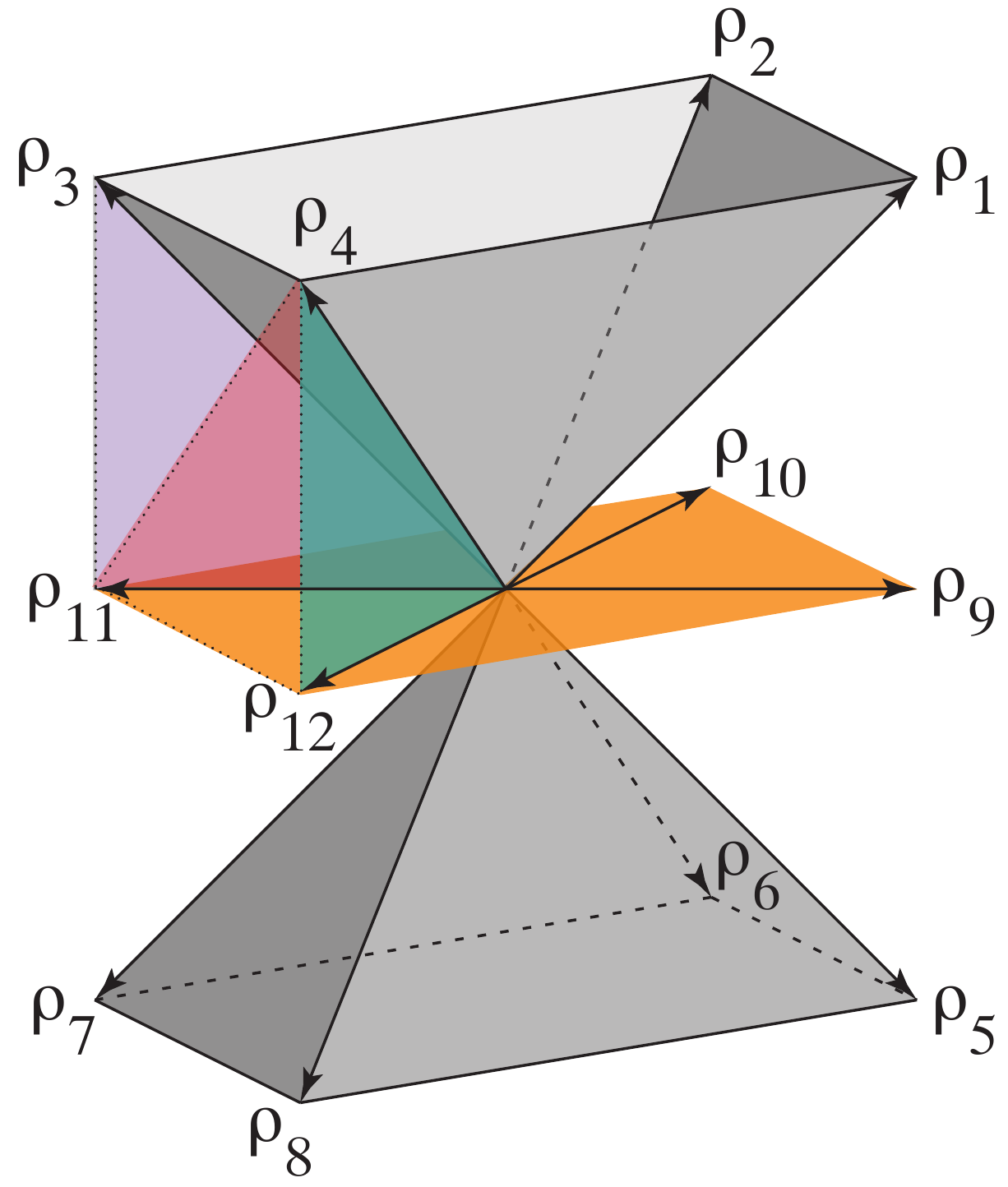}
\caption{The fan $\Sigma$ in $\R^3$ as discussed in Example \ref{example:two-isolated-singular-cones}. On the left we indicate the rays of $\Sigma$, two 3-cones and four 2-cones in the $xy$-plane. From this picture, the idea is to then produce a complete fan by adding smooth cones; the start of this process is indicated on the right.\label{figure:2-distant-singular-cones}}
\end{figure}

\begin{example}Consider the variation $\Sigma'$ on the fan $\Sigma$ of Example \ref{example:two-isolated-singular-cones} as follows. The rays of $\Sigma'$ are precisely the rays $\rho_1,\ldots,\rho_{12}$ of $\Sigma$, and we start as in the picture on the left in Example \ref{example:two-isolated-singular-cones}. However, rather than adding smooth cones, we add non-simplicial cones, each with four rays. If the red cone in the picture on the right in Figure \ref{figure:2-distant-singular-cones} were removed, the result would show the start of the process. Thus $\Sigma'$ has ten maximal cones, viz
\[
\sigma_1 = \langle \rho_1,\rho_4,\rho_9,\rho_{12}\rangle, 
\sigma_2 = \langle \rho_1,\rho_2,\rho_9,\rho_{10}\rangle, 
\sigma_3 = \langle \rho_2,\rho_3,\rho_{10},\rho_{11}\rangle, 
\sigma_4 = \langle \rho_3,\rho_4,\rho_{11},\rho_{12}\rangle, 
\]
\[
\sigma_5 = \langle \rho_5,\rho_8,\rho_9,\rho_{12}\rangle, 
\sigma_6 = \langle \rho_5,\rho_6,\rho_9,\rho_{10}\rangle, 
\sigma_7 = \langle \rho_6,\rho_7,\rho_{10},\rho_{11}\rangle, 
\sigma_8 = \langle \rho_7,\rho_8,\rho_{11},\rho_{12}\rangle, 
\]
\[
\sigma_9 = \langle \rho_1,\rho_2,\rho_3,\rho_4\rangle, 
\sigma_{10} = \langle \rho_7,\rho_8,\rho_9,\rho_{10}\rangle.
\]
\noindent One checks that $\Sigma'$ is a singular, complete, non-simplical, polytopal fan with isolated singular cones. However, the singular cones are {\em not} distant, so Theorem \ref{thm:distant} does not apply. Although we cannot say anything about $K^*_T(X_{\Sigma'})$, the fan $\Sigma'$ is of interest because it is a complete fan, and yet {\em every} maximal cone is an isolated singular cone because each proper face of each maximal cone is smooth.
\end{example}

\begin{example}
Toric degeneration is an important technique in algebraic geometry and representation theory.  It is a technique that starts with a variety and produces a family of varieties, at least one of which is a toric variety. One then hopes to use toric techniques on the toric variety to answer questions about the original variety. The Gelfand-Tsetlin system was first developed in the context of integrable systems \cite{gs:g-ts}. The connection to algebro geometric toric degeneration is described in \cite{nnu:g-ts}. In the simplest non-trivial example, the toric degeneration $X$ of $\mathscr{F}\ell ags(\C^3)$, has fan $\Sigma$ with a single (hence distant) singular cone.  This fan was explicitly described in \cite[\S 3.3]{pab:toric-degen}, and has rays
\[
\rho_1=\left( \begin{array}{r}-1\\0\\0\end{array}\right), \rho_2=\left( \begin{array}{r}1\\0\\0\end{array}\right), 
\rho_3=\left( \begin{array}{r}0\\-1\\0\end{array}\right), \rho_4=\left( \begin{array}{r}0\\1\\0\end{array}\right), 
\rho_5=\left( \begin{array}{r}1\\0\\-1\end{array}\right), \mbox{ and } \rho_6=\left( \begin{array}{r}0\\-1\\1\end{array}\right).
\]
It has maximal cones
\begin{gather*}
\sigma_1=\langle \rho_1,\rho_3,\rho_5\rangle , \sigma_2=\langle \rho_1,\rho_3,\rho_6\rangle , 
\sigma_3=\langle \rho_1,\rho_4,\rho_5\rangle , \sigma_4=\langle \rho_1,\rho_4,\rho_6\rangle ,\\
\sigma_5=\langle \rho_2,\rho_3,\rho_5,\rho_6\rangle , \sigma_6=\langle \rho_2,\rho_4,\rho_5\rangle , 
 \mbox{ and } \sigma_7=\langle \rho_2,\rho_4,\rho_6\rangle .
\end{gather*}
One checks that $\Sigma$ is a singular, complete, non-simplical, polytopal fan with distant singular cones; the single isolated singularity corresponds to the distant singular cone $\sigma_5$.  The geometry of this particular singularity is precisely that of \cite[Example 1.1.18]{colisc:tv}). Theorem~\ref{thm:distant} guarantees that this variety has $K^*_T(X)\cong P_K(\Sigma)$ as $K^*_T$-algebras.
\end{example}

\appendix
\section{The Lattice Ideal Lemma}\label{section:lil}

\noindent In Section \ref{section:R^2-complete} above, we associated to a lattice $L\leq\Z^s$ for some $s>0$, the \ouremph{lattice ideal} $J_L$ in the Laurent polynomial ring $\Z[x_1^{\pm1},\dots,x_s^{\pm1}]$. We also write $J_L$ for the lattice ideal of $L$ in the polynomial ring $\Z[x_1,\dots,x_s]$, where
\[
J_L=\langle x^u-x^v\mid u-v\in L, u,v\in\N^s\rangle.
\]
We write $J_L$ for the lattice ideal in either ring in what follows, taking care to be clear of the context.

If $\mathbb{L}=(\ell_{ij})$ is a matrix whose columns $\ell^1,\ldots,\ell^r$ are a $\Z$-basis for $L$, we write 
\[
J_\mathbb{L}=\left\langle \left(\prod\limits_{i~\text{with}~\ell_{ij}>0}x_i^{\ell_{ij}}\right)-\left(\prod\limits_{i~\text{with}~\ell_{ij}<0}x_i^{-\ell_{ij}}\right)\Bigg|\ 1\leq j\leq r\right\rangle.
\]
\noindent For notational convenience we write $\ell^j=\ell^j_+-\ell^j_-$ where the $i^{th}$ entry of $\ell^j_\epsilon\in\Z^s$ is $\epsilon\ell_{ij}$, for $\epsilon =+,-$. We may then write expressions in the form $x^{\ell^j_+}-x^{\ell^j_-}$.  Again,
$J_\mathbb{L}$  is an ideal in the polynomial ring $\Z[x_1,\dots,x_s]$ or in the Laurent polynomial
ring $\Z[x_1^{\pm1},\dots,x_s^{\pm1}]$, depending on the context. In the Laurent polynomial ring, it is clear that
\[
J_\mathbb{L}=
\left\langle 1-{x}^{\ell_j}~|~1\leq j\leq r\right\rangle,
\]
and so our notation here is consistent with that in Section \ref{section:R^2-complete}.

\begin{lemma}[The lattice ideal lemma for polynomial rings]
In the polynomial ring $\Z[x_1,\dots,x_s]$,
$J_L=J_\mathbb{L}{:}(x_1\cdots x_s)^\infty$, where the latter ideal is the \ouremph{saturation} of $J_\mathbb{L}$ with respect to $x_1\cdots x_s$, ie
\[
J_\mathbb{L}{:}(x_1\cdots x_s)^\infty=\left\{f\in\Z[x_1,\ldots,x_s]\ \bigg|\ f\cdot (x_1\cdots x_s)^N\in J_\mathbb{L}\mbox{ for some }N>0\right\}.
\]
\end{lemma}

\noindent This is \cite[Lemma 7.6]{ms:cca}, but as the full details of the proof are not given explicitly in \cite{ms:cca}, we provide them here.

\begin{proof}
Clearly $J_\mathbb{L}\leq J_L$ and hence $J_\mathbb{L}{:}(x_1\cdots x_s)^\infty\leq J_L$. For the converse we take a generator $x^u-x^v$ for $J_L$, so $u,v\in\N^s$ and $u-v\in L$. We shall show that $x^{u-v}-1\in J_\mathbb{L}{:}(x_1\cdots x_s)^\infty$.

Write $u-v=\sum\limits_{i=1}^ra_i\ell^i$ with $a_i\in\Z$. Then
\[
x^{u-v}-1=\prod\limits_{a_i>0}\left(\frac{x^{\ell^i_+}}{x^{\ell^i_-}}\right)^{a_i}\prod\limits_{a_i<0}\left(\frac{x^{\ell^i_-}}{x^{\ell^i_+}}\right)^{-a_i}-1
\]
and working with the saturation allows us to essentially clear denominators: for some $N>0$,
\[
x^N(x^{u-v}-1)
=
m\left(
\prod\limits_{a_i>0}\left(x^{\ell^i_+}\right)^{a_i}
\prod\limits_{a_i<0}\left(x^{\ell^i_-}\right)^{-a_i}
-
\prod\limits_{a_i>0}\left(x^{\ell^i_-}\right)^{a_i}
\prod\limits_{a_i<0}\left(x^{\ell^i_+}\right)^{-a_i}
\right),
\]
where $m$ is some monomial. It now suffices to show that
\begin{equation}\label{equation:yuk}
\prod\limits_{a_i>0}\left(x^{\ell^i_+}\right)^{a_i}
\prod\limits_{a_i<0}\left(x^{\ell^i_-}\right)^{-a_i}
-
\prod\limits_{a_i>0}\left(x^{\ell^i_-}\right)^{a_i}
\prod\limits_{a_i<0}\left(x^{\ell^i_+}\right)^{-a_i}
\end{equation}
is in $J_\mathbb{L}$, which we do by expressing it in terms of the generators of $J_\mathbb{L}$. We induct on the number of basis elements $\ell^1,\ldots,\ell^r$ involved in (\ref{equation:yuk}).

If only one basis element is involved, say $\ell^1$, and when $a_1>0$ the expression (\ref{equation:yuk}) may be written
\begin{equation}\label{equation:initial-step}
x^{a_1\ell^1_+}-x^{a_1\ell^1_-}=\left(x^{\ell^1_+}-x^{\ell^1_-}\right)\left((x^{\ell^1_+})^{a_1-1}+(x^{\ell^1_+})^{a_1-2}x^{\ell^1_-}+\cdots+x^{\ell^1_+}(x^{\ell^1_-})^{a_1-2}+(x^{\ell^1_-})^{a_1-1}\right)
\end{equation}
and hence is a polynomial multiple of $x^{\ell^1_+}-x^{\ell^1_-}$. The case $a_1<0$ is dealt with similarly. This completes the initial step of the induction.

Now suppose we may write (\ref{equation:yuk}) in terms of the generators of $J_\mathbb{L}$ whenever (\ref{equation:yuk}) involves no more that $k-1$ basis elements $\ell^1,\ldots,\ell^r$, and consider the situation in which $k$ basis elements are involved. Without loss, we may assume $\ell^1$ is involved. We also assume $a_1>0$, as the case $a_1<0$ is similar. Then (\ref{equation:yuk}) may be written as
\begin{eqnarray}
\nonumber&&\left(x^{a_1\ell^1_+}\prod\limits_{a_i>0,i\neq1}x^{a_i\ell^i_+}\prod\limits_{a_i<0}x^{-a_i\ell^i_-}\right)
-
\left(x^{a_1\ell^1_-}\prod\limits_{a_i>0,i\neq1}x^{a_i\ell^i_+}\prod\limits_{a_i<0}x^{-a_i\ell^i_-}\right)
\\
&&+
\left(x^{a_1\ell^1_-}\prod\limits_{a_i>0,i\neq1}x^{a_i\ell^i_+}\prod\limits_{a_i<0}x^{-a_i\ell^i_-}\right)
-\left(\prod\limits_{a_i>0}x^{a_i\ell^i_-}\prod\limits_{a_i<0}x^{-a_i\ell^i_+}\right)\label{equation:more-yuk}
\end{eqnarray}

The first two terms may be written as
\[
(x^{a_1\ell^1_+}-x^{a_1\ell^1_-})\prod\limits_{a_i>0,i\neq1}x^{a_i\ell^i_+}\prod\limits_{a_i<0}x^{-a_i\ell^i_-},
\]
which is, by (\ref{equation:initial-step}), a polynomial multiple of $x^{\ell^1_+}-x^{\ell^1_-}$. It now suffices to consider the final two terms of (\ref{equation:more-yuk}). Observe that
\begin{eqnarray*}
&&
\left(x^{a_1\ell^1_-}\prod\limits_{a_i>0,i\neq1}x^{a_i\ell^i_+}\prod\limits_{a_i<0}x^{-a_i\ell^i_-}\right)
-\left(\prod\limits_{a_i>0}x^{a_i\ell^i_-}\prod\limits_{a_i<0}x^{-a_i\ell^i_+}\right)\\
&=&
\left(x^{a_1\ell^1_-}\prod\limits_{a_i>0,i\neq1}x^{a_i\ell^i_+}\prod\limits_{a_i<0}x^{-a_i\ell^i_-}\right)
-\left(x^{a_1\ell^1_-}\prod\limits_{a_i>0,i\neq1}x^{a_i\ell^i_-}\prod\limits_{a_i<0}x^{-a_i\ell^i_+}\right)\\
&=&
x^{a_1\ell^1_-}\left(\left(\prod\limits_{a_i>0,i\neq1}x^{a_i\ell^i_+}\prod\limits_{a_i<0}x^{-a_i\ell^i_-}\right)
-\left(\prod\limits_{a_i>0,i\neq1}x^{a_i\ell^i_-}\prod\limits_{a_i<0}x^{-a_i\ell^i_+}\right)\right).
\end{eqnarray*}
Ignoring the factor $x^{a_1\ell^1_-}$, the remainder of the final line is of the form (\ref{equation:yuk}), but involving only $k-1$ basis elements. Hence by the inductive hypothesis, it may be expressed as a polynomial multiple of generators of $J_\mathbb{L}$. This completes the inductive step, and the proof.
\end{proof}

We note that in the Laurent polynomial ring
$\Z[x_1^{\pm1},\dots,x_s^{\pm1}]$, ideals are invariant under taking
saturation with respect to $x_1\cdots x_s$. Thus, we have an immediate corollary.

\begin{corollary}[The lattice ideal lemma for Laurent polynomial rings] \label{cor:lattice LPL}
In the Laurent polynomial ring $\Z[x_1^{\pm1},\dots,x_s^{\pm1}]$, we have $J_L=J_\mathbb{L}$.
\end{corollary}

Corollary~\ref{cor:lattice LPL} guarantees the following important relationship between sublattices and the corresponding lattice ideals when working in Laurent polynomial rings; the analogous statement for polynomial rings does not hold.

\begin{proposition}\label{prop:lattice subset}
Let $L$ and ${L'}$ be sublattices of $\Z^n$.  Then working in the Laurent polynomial ring 
$\Z[x_1^{\pm1},\dots,x_s^{\pm1}]$, we have
\[
L\leq {L'} \Longleftrightarrow J_L\leq J_{{L'}}.
\]
\end{proposition}

\begin{proof}
In fact, we shall show that 
\[
L\leq {L'} \Longleftrightarrow J_L\leq J_{{L'}}\mbox{ in }\Z[x_1^{\pm1},\dots,x_s^{\pm1}]\Longleftrightarrow J_L\leq J_{{L'}}\mbox{ in }\C[x_1^{\pm1},\dots,x_s^{\pm1}].
\]

First, if $L\leq L'$, it follows immediately from the definition of a lattice ideal that $J_L\leq J_{L'}$ in $\Z[x_1^{\pm1},\dots,x_s^{\pm1}]$.

Second, suppose that $J_L\leq J_{L'}$ in $\Z[x_1^{\pm1},\dots,x_s^{\pm1}]$. Consider $J_L$ in $\C[x_1^{\pm1},\dots,x_s^{\pm1}]$; each of its generators is of the form $x^u-x^v$ where $u,v\in\Z^s$ and $u-v\in L$. But this is precisely the form of a generator of $J_L$ in $\Z[x_1^{\pm1},\dots,x_s^{\pm1}]$, so by supposition, $x^u-x^v\in J_{L'}$ in $\Z[x_1^{\pm1},\dots,x_s^{\pm1}]$. This means that $u-v\in L'$ and hence $x^u-x^v\in J_{L'}$ in $\C[x_1^{\pm1},\dots,x_s^{\pm1}]$. It follows that $J_L\leq J_{L'}$ in $\C[x_1^{\pm1},\dots,x_s^{\pm1}]$.

Third, and to complete the proof, we shall show that $J_L\leq J_{L'}$ in $\C[x_1^{\pm1},\dots,x_s^{\pm1}]$ implies $L\leq L'$. So suppose that $J_L\leq J_{L'}$ in $\C[x_1^{\pm1},\dots,x_s^{\pm1}]$. Since $J_L$ is a binomial ideal in a Laurent polynomial ring, we may choose a generating set with generators of the form $1-{x}^{\ell_j}$ for $j=1,\dots,r$, where $\ell_j\in\Z^s$. This in turn allows us to define a lattice in $\Z^s$, namely $\mathscr{L}_{J_L} := \mathrm{Span}(\ell_1,\dots,\ell_r)\leq \Z^s$. But by \cite[Theorem 2.1(a)]{ES}, this new lattice must in fact be the original,  $\mathscr{L}_{J_L}=L$. Now, since $J_L\leq J_{L'}$ in $\C[x_1^{\pm1},\dots,x_s^{\pm1}]$, we can choose our generating set of $J_L$ to be a subset of a generating set for $J_{L'}$. This guarantees that $\mathscr{L}_{J_L} \leq \mathscr{L}_{J_{L'}}$, which exactly means $L\leq L'$, as desired.
\end{proof}

\newpage

\section{{\tt Macaulay2} code for Example \ref{example:square based pyramid}}\label{M2 code}

\begin{verbatim}
restart
loadPackage "Polyhedra"
loadPackage "NormalToricVarieties"

--We give the rays as matrices; note they are columns
R1= matrix {{1},{0},{1}}
R2= matrix {{0},{1},{1}}
R3= matrix {{-1},{0},{1}}
R4= matrix {{0},{-1},{1}}
R5= matrix {{0},{0},{-1}}

--Set up the maximal cones, using posHull
C1 = posHull {R1,R2,R3,R4}
C2 = posHull {R1,R2,R5}
C3 = posHull {R1,R4,R5}
C4 = posHull {R2,R3,R5}
C5 = posHull {R3,R4,R5}

--Create the fan
F=fan C1
F=addCone(C2,F)
F=addCone(C3,F)
F=addCone(C4,F)
F=addCone(C5,F)

--Check whether maximal cones are smooth
isSmooth(C1)
isSmooth(C2)
isSmooth(C3)
isSmooth(C4)
isSmooth(C5)

--Verify that intersections of maximal cones with singular C1 are smooth
C12=intersection(C1,C2)
C13=intersection(C1,C3)
C14=intersection(C1,C4)
C15=intersection(C1,C5)
isSmooth(C12)
isSmooth(C13)
isSmooth(C14)
isSmooth(C15)

--Check things about the fan
isSmooth(F)
isComplete(F)
isSimplicial(F)
isPolytopal(F) 
\end{verbatim}


\begin{thebibliography}{99}

\bibitem{anpa:okt} Dave Anderson and Sam Payne,
\newblock Operational {$K$}-theory.
\newblock {\em Doc. Math.} 20:357--399, 2015.


\bibitem{AtiSeg04} Michael F. Atiyah and Graeme Segal,
\newblock Twisted $K$-theory. 
\newblock {\em Ukr. Mat. Visn.} 1 (2004), no. 3, 287--330; translation 
in {\em Ukr. Math. Bull.} 1 (2004), no. 3, 291--334.

\bibitem{BFR} Anthony Bahri, Matthias Franz and Nigel Ray, 
\newblock The equivariant cohomology ring of weighted projective space. 
\newblock {\em Math. Proc. Cambridge Philos. Soc.} 146 (2009), no. 2, 395--405.

\bibitem{BNSS}
Anthony Bahri, Dietrich Notbohm, Soumen Sarkar, and Jongbaek Song,
\newblock On integral cohomology of certain orbifolds.
\newblock {\tt arXiv:1711.01748}.

\bibitem{BSS}
Anthony Bahri, Soumen Sarkar and Jongbaek Song,
\newblock On the integral cohomology ring of toric orbifolds and singular toric
  varieties.
\newblock {\em Algebr. Geom. Topol.}, 17(6):3779--3810, 2017.

\bibitem{bri:spa} Michel Brion,
\newblock The structure of the polytope algebra. 
\newblock {\em Tohoku Math. J.} (2) 49 (1997), no. 1, 1--32.

\bibitem{colisc:tv} David A Cox, John B Little, and Henry K Schenck,
\newblock \emph{Toric Varieties}.
\newblock Graduate Studies in Mathematics Volume 124.
\newblock American Mathematical Society, 2011.

\bibitem{dup:fpv} Delphine Dupont, 
\newblock Faisceaux pervers sur les vari\'{e}t\'{e}s toriques lisses. 
\newblock {\em C. R. Math. Acad. Sci. Paris}, 348(15-16):853--856, 2010.

\bibitem{ES} David Eisenbud and Bernd Sturmfels,
\newblock Binomial ideals.
\newblock {\em Duke Math. J.}, 84(1):1--45, 1996.

\bibitem{MF} Matthias Franz,
\newblock Describing toric varieties and their equivariant cohomology.
\newblock {\em Colloq. Math.} 121 (2010), no. 1, 1--16.

\bibitem{ful:itv} William Fulton,
\newblock {\em Introduction to Toric Varieties}. 
\newblock Annals of Mathematics Studies 131, Princeton University 
Press, 1993.

\bibitem{gs:g-ts} Victor Guillemin and Shlomo Sternberg,
\newblock The Gel'fand-Cetlin system and quantization of the complex flag manifolds. 
\newblock {\em J. Funct. Anal.} 52 (1983), no. 1, 106--128. 

\bibitem{HHRW} Megumi Harada, Tara Holm, Nigel Ray and Gareth Williams,
\newblock The equivariant {$K$}-theory and cobordism rings of divisive weighted projective spaces.
\newblock {\em Tohoku Math. J. (2)} 68(4):487--513, 2016.

\bibitem{HaL} Megumi Harada and Gregory Landweber,
\newblock The {$K$}-theory of abelian symplectic quotients. 
\newblock {\em Mathematical Research Letters} 15(1):57--72, 2008.

\bibitem{kan:ctp} Tamafumi Kaneyama,
\newblock Torus-equivariant vector bundles on projective spaces. 
\newblock {\em Nagoya Math. J.} 111 (1988), 25--40.

\bibitem{ak:fwps} Alexander M. Kasprzyk,
\newblock Bounds on fake weighted projective space.
\newblock {\em Kodai Math. J.} 32 (2009), 197--208.

\bibitem{kly:ebt} Alexander A. Klyachko, 
\newblock Equivariant bundles over toric varieties. (Russian)
\newblock \emph{Izv. Akad. Nauk SSSR Ser. Mat.} 53 (1989), no. 5,
1001--1039, 1135; translation in \emph{Math. USSR-Izv.} 35 (1990), 
no. 2, 337--375.

\bibitem{liya:ste} Chen-Hao Liu and Shing-Tung Yau, 
\newblock On the splitting type of an equivariant vector bundle over a toric manifold. 
\newblock {Preprint \tt arXiv:math/0002031}.

\bibitem{ms:cca}Ezra Miller and Bernd Sturmfels,
\newblock {\em Combinatorial Commutative Algebra}. 
\newblock Springer, New York, 2005.

\bibitem{mor:ktt} Robert Morelli, 
\newblock The $K$-theory of a toric variety. 
\newblock {\em Adv. Math.} 100 (1993), no. 2, 154--182.

\bibitem{nnu:g-ts} Takeo Nishinou, Yuichi Nohara, and Kazushi Ueda,
\newblock Toric degenerations of Gelfand--Cetlin systems and potential functions.
\newblock {\em Adv. Math.} 224 (2010), 648--706.

\bibitem{oda:cba} Tadao Oda,
\newblock {\em Convex Bodies and Algebraic Geometry: An Introduction to the Theory of Toric Varieties}. 
\newblock Springer, New York, 1988.

\bibitem{pab:toric-degen} Milena Pabiniak,
\newblock Displacing (Lagrangian) submanifolds in the manifolds of full flags.
\newblock {\em Adv. Geom.} 15 (2015) No.\ 1,101--108.

\bibitem{pay:tvb} Sam Payne,
\newblock Toric vector bundles, branched covers of fans, and the resolution property. 
\newblock {\em J. Algebraic Geom.} 18 (2009), no. 1, 1--36.

\bibitem{Seg68} Graeme Segal, 
\newblock Equivariant {$K$}-theory.
\newblock {\em Publications Math{\'e}matiques de l'Institut des Hautes 
{\'E}tudes Scientifiques} 34:129--151, 1968.

\bibitem{Thomason}
R.~W. Thomason,
\newblock Comparison of equivariant algebraic and topological {$K$}-theory.
\newblock {\em Duke Math. J.}, 53(3):795--825, 1986.

\bibitem{vevi:hak} Gabriele Vezzosi and Angelo Vistoli,
\newblock Higher algebraic {$K$}-theory for actions of diagonalizable 
groups. 
\newblock {\em Invent. Math.} 153 (2003), no. 1, 1--44.
\end{thebibliography}
\end{document}